\documentclass[11pt]{amsart}

\usepackage[margin=3.3cm]{geometry}
\usepackage{amssymb, amsmath}
\usepackage{amscd}
\usepackage[alwaysadjust]{paralist}
\usepackage{hyperref}
\usepackage[capitalize]{cleveref}
\usepackage{enumitem}
\usepackage{esint}
\usepackage{graphicx} 
\usepackage[rgb]{xcolor} 
\usepackage{verbatim}

\usepackage{amsthm}

\numberwithin{equation}{section}
\newtheorem{claim}[equation]{Claim}

\newtheorem{theorem}[equation]{Theorem}

\theoremstyle{definition}

\newtheorem{rem}[equation]{Remark}

\newtheorem*{acknowledge}{Acknowledgments}

\def\IN{\mathbb N}
\def\IR{\mathbb R}

\def\eps{\varepsilon}

\newcommand{\mult}{\operatorname{mult}}

\newcommand{\supp}{\operatorname{supp}}

\begin{document}
\title[Free boundary minimal surfaces in {E}uclidean balls]{Free boundary minimal surfaces of any topological type in {E}uclidean balls via shape optimization}

\author{Henrik Matthiesen}
\address{Henrik Matthiesen: Department of Mathematics, University of Chicago,
5734 S. University Ave, Chicago, Illinois 60637}
\email{hmatthiesen@math.uchicago.edu}

\author{Romain Petrides}
\address{Romain Petrides, Universit\'e de Paris, Institut de Math\'ematiques de Jussieu - Paris Rive Gauche, b\^atiment Sophie Germain, 75205 PARIS Cedex 13, France}
\email{romain.petrides@imj-prg.fr}
\date{\today}

\maketitle

\begin{abstract} 
For any compact surface $\Sigma$ with smooth, non-empty boundary, we construct a free boundary minimal immersion into a Euclidean ball $\mathbb{B}^N$ where $N$ is controlled in terms of the topology of $\Sigma$.
We obtain these as maximizing metrics for the isoperimetric problem for the first non-trivial Steklov eigenvalue.
Our main technical result concerns asymptotic control on eigenvalues in a delicate glueing construction which allows us to prove the remaining spectral gap conditions to complete the program by Fraser--Schoen and the second named author to obtain such mazimizing metrics.
Our construction draws motivation from earlier work by the first named author with Siffert on the corresponding problem in the closed case.
\end{abstract}

\section{Introduction}

Minimal surfaces naturally appear considering soap films: for instance, in the classical Plateau problem asking for area-minimizing disks whose boundary is a closed curve in $\mathbb{R}^3$. 
After this problem was independently solved by Douglas and Rad\'o, Courant \cite{courant} generalized this question, looking for disks minimizing the area, letting the boundary lie in a constraint surface of $\mathbb{R}^3$. 
This created a lot of activity around so-called free boundary minimal surfaces (see the surveys by Hildebrandt \cite{Hil} and M.~C.~Li \cite{li}). 
In the current paper, we focus on free boundary minimal surfaces in Euclidean unit balls.

In their celebrated, pioneering paper \cite{fs}, Fraser and Schoen made a one to one link between free boundary minimal 
immersions of a surface with boundary into a Euclidean unit ball and critical metrics of Steklov eigenvalues on this surface 
among metrics with boundary of unit length. 
They were inspired by the seminal work by Nadirashvili \cite{nadirashvili} and then El-Soufi and Ilias \cite{Ilias_ElSoufi}, who notably gave the one to one link between critical metrics for Laplace eigenvalues on closed surfaces among metrics with unit area and minimal immersions into a round sphere. 
Since then, the topic of free boundary minimal surface has gained more attention again.
In particular, uniqueness questions and construction of examples for a large variety of topologies have been studied extensively in recent years.
The work by Fraser and Schoen \cite{fs} and then extended by the second author \cite{petrides,petrides-2,petrides-3} gave a natural program for the construction of minimal surfaces by solving isoperimetric optimization problems for eigenvalues.
This has been a permanent source of inspiration for the research of the authors, and also for the current paper. 
Our main result completes the existence question of free boundary minimal immersion into a Euclidean unit ball, for any topology of the surface.

\begin{theorem} \label{thm_min_surf}
Let $\Sigma$ be a compact surface with non-empty boundary. 
Then there is $N \geq 3$ depending on the topology of $\Sigma$ and a branched free boundary minimal immersion $\Phi \colon \Sigma \to \mathbb{B}^N$.
\end{theorem}

Notice that in our result, the dimension of the target ball is controlled by the multiplicity of the first Steklov eigenvalue associated to the pull-back of the Euclidean metric along $\Phi$. 
This multiplicity is controlled in terms of the topology of $\Sigma$ (see \cite{KKP}). 
Beyond the work initiated by Fraser and Schoen \cite{fs}, there are by now plenty of other constructions of free boundary minimal surfaces. 
Using perturbation techniques, Folha, Pacard, and Zolotareva \cite{FPZ} obtained the existence of examples in $\mathbb{B}^3$ with genus $0$ and $1$ and $k$ boundary components for $k$ large. 
Using an equivariant version of min-max theory, Ketover obtained the existence of free boundary minimal surfaces in $\mathbb{B}^3$ of unbounded genus and three boundary components \cite{ketover_1,ketover_2}.
Examples of the same topological type using desingularization techniques were found by Kapouleas and Li \cite{kl}. 
Examples with high genus and connected boundary were constructed by Kapouleas and Wiygul. 
Another recent result by Carlotto, Franz and Schulz \cite{CFS} gives existence of free boundary minimal surfaces with arbitrary genus, connected boundary, and dihedral symmetry.

We obtain \cref{thm_min_surf} by completely resolving the existence question in the isoperimetric problem for the first non-trivial Steklov eigenvalue on compact surfaces wit non-empty boundary.

Recall that for a compact Riemannian surface $(\Sigma,g)$ with non-empty boundary the first non-trivial Steklov eigenvalue $\sigma_1(\Sigma,g)$ is the smallest non-zero eigenvalue of the Dirichlet-to-Neumann operator given by $Tu  =\partial_\nu \hat u$, where $\hat u \in C^\infty(\Sigma)$ denotes the harmonic extension of $u \in C^\infty(\partial \Sigma)$, and $\nu$ is the outward pointing normal field along $\partial \Sigma$.
We also write $L_g(\partial \Sigma)$ for the length of the boundary.

\begin{theorem} \label{thm_max}
Let $\Sigma$ be a compact surface with non-empty boundary.
Then there is a smooth metric $g$ on $\Sigma$ such that
$$
\sigma_1(\Sigma,g) L_g(\partial \Sigma) \geq \sigma_1(\Sigma,h) L_h(\partial \Sigma)
$$
for any smooth metric $h$ on $\Sigma$.
\end{theorem}

This generalizes and also reproves a result due to Fraser and Schoen if $\Sigma$ is orientable and has genus $0$, \cite{fs}. 
In fact, the proof of \cite[Proposition 4.3]{fs} (a special case of \cref{thm_glue} below) appears not to be complete, cf.\ \cite[Remark 1.5 and Appendix A]{GL}.
The analogous result for the first eigenvalue of the Laplace operator on closed surfaces is known by work of the second named author \cite{petrides} and the first named author with Siffert \cite{MS3}.
Very recently, in their interesting paper \cite{KS} Karpukhin and Stern have obtained some related results.
They show that for fixed genus $\gamma$ there is an infinite number of $b \in \IN$ such that \cref{thm_max} holds if $\Sigma$
has genus $\gamma$ and $b$ boundary components.
Their argument relies on a comparison result between Steklov and Laplace eigenvalues combined with the main result of \cite{MS3}.

The main result of the second named author in \cite{petrides-3} applied to the first non zero Steklov eigenvalue $\sigma_1$ 
states that if we set 
$$
\sigma_1(\gamma,k) = \sup_{g} \sigma_1(\Sigma,g) L_g(\partial \Sigma)
$$
for an orientable surface $\Sigma$ of genus $\gamma$ with $k$ boundary components, and if one has that
$$ 
\sigma_1(\gamma,k) > \sigma_1(\gamma-1,k+1) \hbox{ for } \gamma\geq 1 \hbox{ and } k\geq 1
$$ 
and
$$ 
\sigma_1(\gamma,k) > \sigma_1(\gamma,k-1) \hbox{ for } \gamma\geq 0 \hbox{ and } k\geq 2 \hskip.1cm,
$$
then $\sigma_1(\gamma,k)$ is achieved by a smooth metric. 
\cref{thm_max} then follows by induction from the following glueing theorem which is our main technical result.  

\begin{theorem} \label{thm_glue}
Let $(\Sigma,g)$ be a compact surface with smooth, non-empty boundary.
Suppose that $\Sigma'$ is topologically obtained from $\Sigma$ by attaching a strip along two opposite sides of its boundary along two disjoint portions of the boundary of $\Sigma$.
Then there is a smooth metric $g'$ on $\Sigma'$ such that
$$
\sigma_1(\Sigma',g') L_g'(\partial \Sigma') > \sigma_1(\Sigma,g) L_g(\partial \Sigma)
$$ 
\end{theorem}

This result gives the required gaps as soon as $\sigma_1(\gamma-1,k+1)$ and $\sigma_1(\gamma,k-1)$ are achieved by a smooth metric. 
We explain this in more detail in \cref{topchange}. 
As mentioned above, by induction and a combination of \cite{petrides-3} and \cref{thm_glue}, also using that the flat disk achieves $\sigma_1(0,1)$, we obtain \cref{thm_max}. 
While it is not written in \cite{petrides-3}, the non-orientable case follows along the very same lines. 
We refer to \cite{MS1} for more details on the non-orientable closed case. 
\cref{thm_glue} is the analogue of \cite[Theorem 1.3]{MS3} in the Steklov case. 
In the closed case we attach a cylinder along two small boundaries of removed disks on the original surface $\Sigma$ of genus $\gamma$. 
Let us now describe some features of the proof of \cref{thm_glue}.

As discussed in more detail in \cite{MS2,MS3} there are some serious obstructions on glueing constructions for which one can hope to obtain the monotonicity result from Theorem \ref{thm_glue}. 
Indeed, the most natural idea to prove this result in the closed case was to attach a thin flat cylinder of 
length $L$ and radius $\eps$ to a surface $\Sigma$ of genus $\gamma$ along the boundary of two removed disks of radius $\eps$, coming from \cite{ces} 
(they used the result from \cite{anne}). 
It is then natural but much harder to compute the first non-zero term in the asymptotic expansion of the first non-zero eigenvalue on the 
perturbed surface $\Sigma_{\eps}$ of genus $\gamma +1$, as $\eps\to 0$. 
Of course, the first eigenvalue on $\Sigma_{\eps}$ might very well be smaller than
the first eigenvalue on $\Sigma$.
But one can hope that the positive extra-term of size $\eps$ given by the asymptotic expansion of the area compensates this loss as $\eps \to 0$. 
For deep reasons, this is not always possible. 
One expects the range of parameters $L$ for which there is hope to get sufficiently strong asymptotic control on the eigenvalue to be such that the perturbed surface $\Sigma_\eps$ enjoys some interaction between the spectra on the thick and the thin part, respectively.

More precisely, one 
should adjust the length of the cylinder $L$ (potentially depending on $\eps$) so that the first eigenvalue of the interval of length $L$ is close to the first eigenvalue of the thick part to observe this interaction phenomenon.
However, there is a fatal obstruction term in the asymptotic expansion of the eigenvalue containing $u_{\star}(p_0) + u_{\star}(p_1)$ where $u_{\star}$ is a first eigenfunction of $\Sigma$
and the handle is attached near the points $p_0$ and $p_1$.
 If this term does not vanish, one can never obtain the strict inequality from \cref{thm_glue} by this technique at least for such parameters $L$, 
 but other parameters are not expected to have a chance anyways.
This can occur when the first eigenvalue has multiplicity on $\Sigma$, which is exactly the situation we have for a maximal metric. 

This remark also applies to the Steklov spectrum (see \cite{fs2} for similar constructions): 
The natural strip that one would like to use is a flat rectangle of size $L\eps\times \eps^2$ that we attach along intervals of length $\eps^2$ on the boundary of $\Sigma$. 
An analogous analysis gives the obstruction term $u_{\star}(p_0) + u_{\star}(p_1)$ where $u_{\star}$ is a first Steklov eigenfunction of $\Sigma$. 
Notice that in the special case of the sphere or  the disk, respectively, attaching the thin part along antipodal points gives that obstruction term from above vanishes.
But there there is no way to obtain this for instance on the flat equilateral torus (the unique maximizer among tori \cite{nadirashvili}).

Therefore, in order to prove these type of glueing results one has to drastically change the geometry of the attached thin part. 
The guiding principle here, originating from \cite{MS2}, is that the convergence rate can not be any better than the $L^1$-norm of an $L^2$-normalized eigenfunction on the thin part.
(For the flat rectangle this is of size $\eps^{1/2}$ losing against the additional boundary length on scale $\eps$.)
In particular, this forces us to work with a geometry that observes the formation of continuous spectrum in the limit $\eps \to 0$.
Also, we notably introduce a big asymmetry between the attaching boundaries, while we try to have a thin part with computable spectra. 
The glueing construction in \cite{MS3} by the first author jointly with Siffert uses a thin, hyperbolic cusp of area $\eps$ truncated at $R=R(\eps)$, whose Dirichlet and Neumann spectrum is perfectly computable by separation of variables, and has an infinite number of eigenvalues converging to $\frac{1}{4}$ as $\eps \to 0$. 
Again, one can play with a parameter of dilatation $t$, such that $\frac{t^2}{4}$ is the first eigenvalue of the thin part and is close to the first eigenvalue of the thick part, in order to capture the interaction between both spectra, and cleverly choose the parameters $t$ and $R$ so that the expected strict inequality occurs.

Compared to the approach in \cite{MS3} we have to face a number of new difficulties in order to prove the \cref{thm_glue}. 
These are most significantly related to the fact that the Steklov problem 
does not enjoy as strong separation of variables properties as available for the Laplace spectrum. On the other hand, we strongly believe that the techniques developed here - whose necessity grew out of the aforementioned reasons - could be used to shorten the argument in \cite{MS3}.

The domain analogous to hyperbolic cylinders in the Steklov case is a cuspidal domain introduced by Nazarov-Taskinen \cite{NT}, as the region on the plane such that $y>0$ and $-\frac{y^2}{2} \leq x\leq \frac{y^2}{2}$, since it is the simplest example having a continuous spectrum. 
We discovered that there is a strong link between the structure of the spectrum and associated eigenfunctions on a truncated hyperbolic cusp and on these cuspidal domains truncated at $r \leq y\leq 1$, and this is the reason why the approach in the closed case and Steklov case are so related.

\medskip
\subsection*{Organization of the paper}

In \cref{section2}, we explain the glueing construction and the ansatz for the behaviour of solutions to the eigenvalue equation of the Dirichlet-to-Neumann operator on cuspidal domains, 
inspired by the well-known solutions of the eigenvalue equation for the Laplacian on a truncated hyperbolic cusp. 

In \cref{section3}, we give a good upper bound on the eigenvalues $\sigma_{\eps}^i$ for $i\in \{1,\dots,K+1\}$ 
on the glued surface in terms of $\sigma_{\star}$ (the first eigenfunction on the thick surface $\Sigma$, of multiplicity $K$) and parameters on the thin part.

We then give a first pointwise estimate on eigenfunctions in the attaching region (\cref{section4}).

The accurate energy bound deduced from \cref{section3} combined with the results from \cref{section4}, which give some control on the boundary values, allow us to perform an asymptotic analysis on the thin part up to a good choice of dilatation parameter on the cuspidal domain
so that we have strong interaction between the first eigenvalue of the thick part and the bottom of the spectrum of the thin part. 
This fine analysis on the eigenvalue equation on the thin part is performed in \cref{section5}. 
By testing the first eigenfunction on the thick part  
we also deduce that $\sigma_{\eps}^1 - \sigma_{\star} = O(\eps)$ as the perimeter of the cuspidal domain tends to zero on scale $\eps$. 

To conclude, we need to improve this bound to $\sigma_{\eps}^1 - \sigma_{\star} = o(\eps)$ as $\eps\to 0$. 
Fortunately,  we can use one of the eigenfunctions associated to $\sigma_{\eps}^i$ for $i\in \{2,\dots,K+1\}$ 
in order to correct for the non-zero mean value of the first eigenfunction on the thick part
leading to an improved bound.
In \cref{section6}, we use the results from \cref{section2} and the refined energy estimates from \cref{section5} to choose an improved test function for $\sigma_{\star}$ and complete the proof of Theorem \ref{thm_glue}.

\begin{acknowledge}
The first named author is grateful to Iosif Polterovich for pointing out  the work \cite{NT} by Nazarov and Taskinen to him.
He would also like to thank the Universit\'e de Paris for their hospitality during his stay in December 2019.
The second named author would like to thank the University of Chicago for an invitation. 
Both visits initiated the preparation of this paper.
\end{acknowledge}

\section{The glueing construction} \label{section2}

In this section we introduce the two parameter family of competitors for the problem that we use to obtain \cref{thm_glue}.
We also discuss some of the properties of the cuspidal domain that we use.
In particular, we give several analogies regarding its spectral properties with the family of truncated hyperbolic cusps used in the closed case.

\subsection{The construction} 

Let $(\Sigma,g)$ be a compact connected Riemannian surface with a non-empty smooth boundary $\partial \Sigma$ of 
length $L_g(\partial \Sigma)= 1$.

Let $\eps>0$. 
We set 
$$
\Omega_{\eps} = \left\{ (x,y) \in \mathbb{R}^2 ; r_{\eps}\leq y\leq 1 , -\frac{y^2}{2}\leq x \leq \frac{y^2}{2} \right\}
$$ 
for a parameter $0<r_{\eps}<1$ endowed with the metric 
$$
g_{\epsilon} = \frac{\epsilon^{4}dx^2 + \epsilon^{2}dy^2}{t_{\eps}^2}
$$ 
for some $t_{\eps}>0$. 
Note that $(\Omega_{\eps},g_{\eps})$ is isometric to $\{ (x,y) \in \mathbb{R}^2 ; \eps r_{\eps} \leq y\leq \eps , -\frac{y^2}{2}\leq x \leq \frac{y^2}{2} \}$ endowed with the metric $\frac{dx^2 + dy^2}{t_{\eps}^2}$. 
We also write 
$$
I_{\eps}^{\pm} = \left\{ (x,y) \in \mathbb{R}^2 ; r_{\eps}\leq y\leq 1 ,  y = \pm \frac{x^2}{2} \right\}.
$$
We glue the cuspidal domain $\Omega_{\eps}$ at the neighborhood of two points $p_0,p_1 \in \partial\Sigma$ in the following way. 
For $i=0,1$, let $\varphi_i \colon B_i \to \mathbb{D}_2^+$ be a conformal chart at the neighborhood of $p_i \in B_i$ such that $\varphi_i(p_i) = 0$, $\varphi_i(B_i) = \mathbb{D}_2^+$, $\varphi_i(B_i\cap \partial \Sigma) = ]-2,2[\times \{0\}$ and
$$ 
g = \varphi_i^{\star} \left( e^{2\omega_i} \left(dx^2 + dy^2\right) \right)
$$
for a smooth function $\omega_i : \mathbb{D}_2^+ \to \mathbb{R}$. 
We denote by $\Sigma_{\eps}$ the surface (which also depends on our choice of $r$ and $t$ but those in turn will depend on $\eps$) that we obtain by gluing $\Sigma$ and $\Omega_{\eps}$ along the intervals $ [-\frac{1}{2},\frac{1}{2}] \times \{0\}$ in the charts 
$$ f_0  : B_0 \to \mathbb{D}_{\frac{2}{\eps^2}}^+  \text{ and } g_0 : \Omega_{\eps} \to E_0 \hskip.1cm,$$
at the neighborhood of $p_0$, where
\begin{equation} \label{eqdefE0} 
E_0 = \left\{(x,z); \frac{r-1}{\eps} \leq z\leq 0 \text{ and } -\frac{\left(1+\eps z\right)^2}{2} \leq x \leq \frac{\left(1+\eps z\right)^2}{2}\right\} 
\end{equation}
and in the charts 
$$
 f_1  : B_1 \to \mathbb{D}_{\frac{2}{r^2\eps^2}}^+ \text{ and } g_1 : \Omega_{\eps} \to E_1
 $$
at the neighborhood of $p_1$, where 
\begin{equation} \label{eqdefE1}
E_1 =  \left\{(x,z); \frac{r-1}{r^{2}\eps} \leq z\leq 0 \text{ and }-\frac{\left(1-r\eps z\right)^2}{2} \leq x \leq \frac{\left(1-r\eps z\right)^2}{2}\right\} 
\end{equation}
and where $f_i$ and $g_i$ are defined by the formulae
$$ 
f_0^{-1}(x,z) = \varphi_0^{-1}\left(\eps^2\left(x,z\right) \right) \text{ and }  g_0^{-1}(x,z)=(x, 1+\eps z)  \hskip.1cm,
$$
$$
 f_1^{-1}(x,z)=\varphi_1^{-1}\left(\eps^2 r^2\left(x,z\right) \right) \text{ and } g_1^{-1}(x,z)=\left(r^2 x, r \left(1 - r\eps z\right) \right) \hskip.1cm.
 $$
so that in the chart at the neighborhood of $p_0$ in $\Sigma_{\eps}$ the metric is given by
$$ 
\begin{cases} \eps^4 e^{2\omega_0(\eps^2(x,z))} \left(  dx^2 +dz^2 \right) & \text{ if } 0\leq z<\frac{2}{\eps^2}\text{ and }  -\frac{1}{2}\leq x\leq \frac{1}{2}  \\
\frac{\eps^4}{t_{\eps}^2} \left( dx^2 + dz^2 \right) & \text{ if } \frac{r-1}{\eps} \leq z\leq 0  \text{ and } -\frac{\left(1+\eps z\right)^2}{2} \leq x \leq \frac{\left(1+\eps z\right)^2}{2}
\end{cases} 
$$
and in the chart at the neighborhood of $p_1$ in $\Sigma_{\eps}$ the metric is given by
$$
 \begin{cases} r^4\eps^4 e^{2\omega_1(r^2\eps^2(x,z))} \left(  dx^2 +dz^2 \right) & \text{ if } 0\leq z<\frac{2}{\eps^2}  \text{ and }  -\frac{1}{2}\leq x\leq \frac{1}{2} \\
\frac{r^4\eps^4}{t_{\eps}^2} \left( dx^2 + dz^2 \right) & \text{ if } \frac{r-1}{r^{2}\eps} \leq z\leq 0 \text{ and } -\frac{\left(1-r\eps z\right)^2}{2} \leq x \leq \frac{\left(1-r\eps z\right)^2}{2} \hskip.1cm.
\end{cases} 
$$

We denote by $\sigma_{\star}$ the first non-zero Steklov eigenvalue of the surface $\Sigma$.
Moreover, $\sigma^1_{\eps}$ denotes the first non-zero Steklov eigenvalue on $\Sigma_{\eps}$.
We aim at proving that for suitable choices of the parameters $t_{\eps}$ and $r_{\eps}$ we have that
$$
\sigma^1_{\eps} = \sigma_{\star} + o(\eps) \hskip.1cm,
$$ 
where $\eps$ is the scale of the extra length of the boundary when we glue $\Omega_{\eps}$ to $\Sigma$.

We also remark that it is easy to approximate the metric on $\Sigma_\eps$ by smooth metrics e.g.\ using
\cite[Lemma 4.1]{kokarev} combined with the observation that $\Sigma_\eps$ carries
a smooth conformal structure.

\subsection{The topological change } \label{topchange}

Note that the new surface $\Sigma_\eps$ will differ topologically from the initial surface $\Sigma$.
Through different choices of the points $p_i$ and orientations of the charts $f_i,g_i$ this gives rise to various options for the topological type of $\Sigma_\eps$.
None of these choices will affect our analytic arguments at all.
However, it is of fundamental importance for our application towards \cref{thm_max} that this covers all the topological changes required to apply the main result from \cite{petrides-3}.
Let us briefly explain how to achieve this.

We first discuss the case that $\Sigma$ is orientable with genus $\gamma \geq 0$ and $k \geq 1$ boundary components.
If we take $p_0$ and $p_1$ to lie in the same component of the boundary we can obtain two types of surfaces by attaching the cuspidal domain $\Omega_\eps$.
If we choose compatible orientations for the charts $f_i$ and $g_i$, 
$\Sigma_{\eps}$ is orientable has genus $\gamma$ and $k+1$ boundary components.
If we reverse the orientation of one of the charts, the resulting surface will be non-orientable, have $k$ boundary components and non-orientable genus $2\gamma+1$.

If we assume that $k \leq 2$ there is also the option of taking $p_0$ and $p_1$ to lie in different components of $\partial \Sigma$.
In this case, if we attach $\Omega_\eps$ such that we obtain an orientable surface (i.e.\ we choose compatible orientations for the charts), we find that $\Sigma_\eps$
has genus $\gamma+1$ and $k-1$ boundary components.
If we reverse the orientation of one of the charts in which we glue, the new surface will have non-orientable genus $2\gamma+2$ and $k-1$ boundary components.

If we start with a non-orientable surface $\Sigma$ of non-orientable genus $\gamma$ and $k$ boundary components,
the topological type depends only on the location (in the same or in different components of the boundary) of the points $p_0$ and $p_1$.
It will either have non-orientable genus $\gamma$ and $k+1$ boundary components, or
non-orientable genus $\gamma+1$ and $k$ boundary components.

\subsection{Spectral properties of the cuspidal domain}

We briefly discuss some spectral properties of the cuspidal domains $\Omega_\eps$.
While we do not give any proofs here, we 
give some motivation behind our ansatz for the asymptotic expansion.
Later we deal with eigenfunctions on the cuspidal domain that are restrictions of eigenfunctions on $\Sigma_\eps$ and hence obey weaker control on the boundary values.

\subsubsection{The energy on one dimensional functions and the main ansatz} \label{sec_ansatz}

For simplicity we consider $t=1$ for the moment.
First of all, notice that for fixed $y \in [r,1]$ the standard Poincar{\'e} inequality applied along line segments $\{(x,y) \ : -\tfrac{y^2}{2} \leq x \leq \tfrac{y^2}{2}\} $
implies that functions with bounded energy on $\Omega_\eps$ endowed with $g_{\eps}$ become more and more constant in the $x$-direction for $\eps$ small.
Therefore, as a first ansatz, we would like to consider functions $\phi \colon \Omega_\eps \to \IR$ with $\phi(x,y)=\phi(y)$.
Of course, these will never be exact eigenfunctions.
The energy of these functions is given by
$$
\int_{\Omega_\eps} |\nabla \phi|^2\, dA_\eps = \eps \int_{r}^1 y^2 \phi_y^2 \, dy,
$$
while the boundary mass is
$$
\int_{I^+ \cup I^-} |\phi|^2\, dl_\eps = 2\eps (1+O(\eps^2))\int_{r}^1 |\phi|^2 \, dy.
$$
We want to point out that this resembles (up to the error term on the scale $\eps^2$, a factor of $2$, and changing $y \to 1/y$) exactly the Dirichlet energy and the $L^2$-norm, respectively,  on rotationally symmetric functions on a truncated hyperbolic cusp

This suggest to use the following change of variables that is the starting point of our asymptotic analysis.
For a function $\phi \in W^{1,2}(\Omega_\eps)$ we write

\begin{equation} \label{eqdeftheta}
\phi(x,y) = \frac{\sqrt{t_{\eps}} y^{-\frac{1}{2}}}{\sqrt{\eps}\sqrt{\ln\frac{1}{r}}} \theta \left(x, \frac{\ln(y)}{\ln(r)} \right) \hskip.1cm,
 \end{equation}
where
$ \theta $ is defined on 
$$
\widetilde{\Omega} = \left\{ (x,v) \in \mathbb{R}^2 ; 0 \leq v \leq 1 , -\frac{r^{2v}}{2}\leq x \leq \frac{r^{2v}}{2} \right\}.
$$ 
We introduce new coordinates for $\Omega$ by the change of variables $(x,y)=(x,r^v)$.
It is convenient to write, for a function $\theta$ defined on $\widetilde{\Omega}$, for the mean value on horizontal lines
 $$\overline{\theta} = r^{-2v} \int_{-\frac{r^{2v}}{2}}^{\frac{r^{2v}}{2}}  \theta dx \hskip.1cm.$$
We can then write in the new coordinates
\begin{equation} \label{eqthetasquare} 
\int_{I^{\pm} } \phi^2 dl_{\eps} =\int_0^1 \theta^2\left(\pm \frac{r^{2v}}{2},v \right) \sqrt{1+\eps^2r^{2v}} dv \hskip.1cm,  
\end{equation}
\begin{equation} \label{eqthetamean} 
\int_{I^{\pm} } \phi dl_{\eps} = \sqrt{\ln\frac{1}{r}} \frac{\sqrt{\eps}}{\sqrt{t_{\eps}}} \int_0^1 r^{\frac{v}{2}} \theta\left(\pm \frac{r^{2v}}{2},v \right) \sqrt{1+\eps^2r^{2v}} dv,  
\end{equation}
and
\begin{equation}  \label{eqgradtheta} 
\int_{\Omega} \left\vert \nabla\phi \right\vert^2_{g_{\eps}} dA_{\eps} =t_{\eps}\left( \frac{1}{\eps^2} \int_{0}^1 \left( \int_{-\frac{r^{2v}}{2}}^{\frac{r^{2v}}{2}}  \theta_x^2 dx \right) dv +  \int_0^1 \overline{ \left( \frac{\theta}{2} + \frac{\theta_v}{\ln \frac{1}{r}} \right)^2 } dv\right) 
\end{equation}
where $dA_{\eps} = \frac{\eps^{3}}{t_{\eps}^2}dxdy$, $\left\vert \nabla\phi \right\vert^2_{g_{\eps}} =t_{\eps}^2\left( \eps^{-4}\phi_x^2 +\eps^{-2}\phi_y^2\right)$ and $dl_{\eps} =\frac{\eps}{t_{\eps}} \sqrt{1+\eps^2y^2}dy$ on $I^{\pm}$, noticing that 
$$
\phi_x = \frac{\sqrt{t_{\eps}}y^{-\frac{1}{2}}}{\sqrt{\eps}\sqrt{\ln\frac{1}{r}}} \theta_x \hspace{0.2cm}  
\text{and} \hspace{0.2cm}
\phi_y = -\frac{\sqrt{t_{\eps}}y^{-\frac{3}{2}}}{\sqrt{\eps}\sqrt{\ln\frac{1}{r}}}  \left( \frac{\theta}{2} + \frac{\theta_v}{\ln\frac{1}{r}}  \right).
$$ 

Most of our arguments will take place on the level of $\theta$ since this is the scale on which we can hope to get good control 
on eigenfunctions.
Broadly speaking we aim at proving that if $\phi$ is an eigenfunction on $\Omega_\eps$, then $\theta$ converges to a solution of $f''+ \nu f =0$ for $\eps \to 0$ in a sufficiently strong sense.

\subsubsection{The asymptotic behaviour of the spectrum for the cuspidal domains}

The argument in \cite{MS3} relies on the following properties of the spectrum of the truncated hyperbolic cusp (with parameter $\alpha \in (1/3,1/2)$) with area on scale $\eps$:

\begin{itemize}
\item The first non-trivial Dirichlet eigenvalue is bounded away from zero.
\item The $L^1$-norm of the normal derivative along the boundary of the first non-trivial $L^2$-normalized eigenfunction is of size $o(\eps)$.
(It is on scale $\eps^{(3\alpha+1)/2}$).
\item The separation of two consecutive Dirichlet eigenvalues is much larger than the $L^1$-norm of the normal derivative along the boundary of the normalized eigenfunctions.
(It is $l \eps^{2\alpha}$ versus $l \eps^{(3\alpha+1)/2}$ for the $l$-th eigenvalue.)
\item
The first non-trivial Neumann and Dirichlet eigenvalues are related by $\lambda_0 \leq \mu_1$.
\end{itemize}

The starting point of our construction was to find an analogue for the Steklov problem.
It turns out that this is given by the cuspidal domains that we use.
In fact, one can prove similar assertions on their spectrum.
However, the more robust asymptotic techniques developed here to attack this problem turn out to apply more generally directly to restrictions of eigenfunctions from $\Sigma_\eps$ using that one can propagate
some control on the boundary values from the compactness of $\Sigma$.

\section{Upper bounds for eigenvalues of the glued surface} \label{section3}

Recall that we denote by $\sigma_\star$ the first non-trivial Steklov eigenvalue of $\Sigma$.
Moreover, we write $K = \mult \sigma_\star$ for the multiplicity of $\sigma_\star$.

Using appropriate extensions of $\sigma_\star$-eigenfunctions and an eigenfunction of the limiting quadratic form on $\Omega_\eps$ we can obtain some upper bounds on $\sigma_\eps^1,\dots,\sigma_\eps^{K+1}$
through a classical test function argument on the variational characterization of eigenvalues.

\begin{claim}
The first  
eigenvalue on $\Sigma_{\eps}$ satisfies
\begin{equation}\label{ineqsigmaeps} 
\sigma_{\eps}^1 \leq \min\left\{ \sigma_{\star} + O\left(\frac{\eps}{\left(\ln\frac{1}{r}\right)^2} + \eps^2 \right),  \frac{t_{\eps}}{8} + \frac{t_{\eps} \pi^2}{2\left(\ln r\right)^2} + O\left(\frac{\eps}{\left(\ln\frac{1}{r}\right)^{3}} + \eps^2\right)  \right\} \hskip.1cm,
\end{equation}
as $\eps\to 0$ and $r\to 0$.
Moreover, we have that
\begin{equation}\label{ineqsigmaeps2} 
\sigma_{\eps}^{K+1} \leq \max\left\{ \sigma_{\star} ,  \frac{t_{\eps}}{8} + \frac{t_{\eps} \pi^2}{2\left(\ln r\right)^2} \right\} + O\left(\frac{\eps}{\ln\frac{1}{r}} + \frac{\eps^{\frac{1}{2}}}{\left(\ln \frac{1}{r}\right)^{\frac{3}{2}}} + \eps^2 \right) \end{equation}
as $\eps\to 0$ and $r\to 0$.
\end{claim} 

\begin{proof}
Let $u_{\star}$ be an eigenfunction associated to the first eigenvalue $\sigma_{\star}$ on $\Sigma$ with unit $L^2$-norm on the boundary. 
We set 
$$ 
u_{\eps} = \begin{cases} \left(u_{\star}(p_1) - u_{\star}(p_0)\right) \frac{\ln y}{\ln r} + u_{\star}(p_0) & \text{in} \  \Omega_{\eps} \\ 
u_{\star} + \eta( \eps^2 \varphi_0 ) \left( u_{\star}(p_0) - u_{\star} \right) + \eta( r^2 \eps^2 \varphi_1) \left(u_{\star}(p_1) - u_{\star} \right) & \text{in} \ \Sigma
  \end{cases}
 \hskip.1cm,
 $$
where $\eta$ is a smooth cut-off function such that $\eta = 1$ on $\mathbb{D}$, $\eta=0$ in $\mathbb{R}^2 \setminus \mathbb{D}_2$ 
and $\nabla \eta$ is a bounded function, so that $u_{\eps}$ is a Lipschitz function on $\Sigma_{\eps}$ which extends
$u_{\star}$ well to $\Omega_{\eps}$ up to slightly modifying $u_{\star}$ at the neighbourhood of $p_0$ and $p_1$.

By the variational characterization of the first eigenvalue we have that
$$
 \sigma_\eps^1 
 \leq 
 \frac{\int_{\Sigma_{\eps}}\left\vert \nabla u_{\eps}\right\vert_{g_{\eps}}^2 dA_{\eps}}{\int_{\partial\Sigma_{\eps}}\left(  u_{\eps} \right)^2 dl_{\eps} - \frac{\left(\int_{\partial\Sigma_{\eps}}u_{\eps}  dl_{\eps}\right)^2}{ 1-2\pi\eps^2 + \frac{2\eps}{t} +O(\eps^3) }} 
 \leq 
 \frac{\int_{\Sigma}\left\vert \nabla u_{\eps} \right\vert_{g}^2 dA_{g} + \int_{\Omega_{\eps}} \left\vert \nabla u_{\eps} \right\vert_{g_{\eps}}^2 dA_{g_{\eps}}   }{\int_{\partial\Sigma}u_{\star}^2 dl_{\eps} - O(\eps^2)} 
 $$
as $\eps \to 0$, where we noticed in the denominator that $\left\vert u_{\star} - u_{\star}(p_0) \right\vert = O(\eps^2)$ and $\left\vert u_{\star} - u_{\star}(p_1) \right\vert = O(r^2 \eps^2)$ in the neighbourhoods of $p_0$ and $p_1$ on which $u_\eps$ and $u_\star$ do not agree. 
We also noticed that $u_{\eps}$ is uniformly bounded on $\Omega_{\eps}$ so that the mean value on the boundary is controlled by the length of order $\eps$ of the boundary. 
For the gradient, we have
$$
 \int_{\Omega_{\eps}} \left\vert \nabla u_{\eps} \right\vert_{g_{\eps}}^2 dA_{g_{\eps}} = \eps \int_{r}^1 y^2 \left\vert \partial_y u_{\eps} \right\vert^2 dy \leq \frac{\eps}{\left(\ln\frac{1}{r}\right)^2}
 $$
and
\begin{eqnarray*}
 \int_{\Sigma}\left\vert \nabla u_{\eps} \right\vert_{g}^2 dA_{g} 
 & \leq & 
 \int_{\Sigma}\left\vert \nabla u_{\star} \right\vert^2 dA_{g} + 2\int_{\Sigma}\left\vert \nabla u_{\star} \right\vert \left\vert \nabla (u_{\star}-u_{\eps}) \right\vert dA_{g}+ \int_{\Sigma}\left\vert \nabla \left(u_{\star}-u_{\eps}\right) \right\vert^2 dA_{g} 
 \\
& \leq & 
\int_{\Sigma}\left\vert \nabla u_{\star} \right\vert_{g}^2 dA_{g} + 2C \eps^2 \left\| \nabla\left( u_{\star}-u_{\eps}\right) \right\|_{L^2} + \left\| \nabla\left(u_{\star}-u_{\eps}\right)  \right\|_{L^2}^2
 \\
& \leq & 
\int_{\Sigma}\left\vert \nabla u_{\star} \right\vert_{g}^2 dA_{g} + O(\eps^4)
\end{eqnarray*} 
as $\eps\to 0$, where we used 
conformality of $\eps^2 \varphi_0$ and $\eps^2\varphi_1$ and
again that $\left\vert u_{\star} - u_{\star}(p_0) \right\vert = O(\eps^2)$ and $\left\vert u_{\star} - u_{\star}(p_1) \right\vert = O(r^2 \eps^2)$ on $\supp \nabla (\eta (\eps^2 \varphi_0))$ and $\supp \nabla (\eta (r^2\eps^2 \varphi_1))$, respectively.
This gives the first term in the right-hand side of the inequality \eqref{ineqsigmaeps}.

Now, we aim at proving the inequality
\begin{equation} \label{eqineqsigmaeps} \sigma_\eps^1 \leq \frac{t_{\eps}}{8} + \frac{t_{\eps} \pi^2}{2 (\ln r )^2}+O\left(\eps\left(\ln\frac{1}{r}\right)^{-3} + \eps^2 \right)    \hskip.1cm. \end{equation}
Indeed, we test the variational characterization of $\sigma_{\eps}$ with the functions 
$$
\phi(y) = \frac{\sqrt{t_{\eps}}y^{-\frac{1}{2}}}{\sqrt{\eps}\sqrt{\ln\frac{1}{r}}} \sin\left(\pi \frac{\ln(y)}{\ln(r)}\right).
$$
That is, thanks to \eqref{eqdeftheta}, we compute the quantities on the right hand sides in \eqref{eqthetasquare} \eqref{eqthetamean} and \eqref{eqgradtheta} for $f(v) = \sin(\pi v)$.
We have that
$$ \int_0^1  \left( \frac{f}{2} + \frac{f'}{\ln\frac{1}{r}} \right)^2  dv = \int_0^1   \frac{f^2}{4}  dv + \int_0^1 \frac{(f')^2}{\left(\ln\frac{1}{r}\right)^2} dv + \int_0^1  \frac{f f'}{\ln\frac{1}{r}}  dv \hskip.1cm,$$
and that
$$ \int_0^1  \frac{f f'}{\ln\frac{1}{r}}  dv = \int_0^1  \frac{\left(f^2\right)'}{2\ln\frac{1}{r}}  dv = \frac{f(1)^2 - f(0)^2}{2\ln\frac{1}{r}} =0 \hskip.1cm.$$
We denote by $I = \int_{0}^1 r^{\frac{v}{2}} \sin(\pi v)dv $. 
Integrating by parts twice, we get
$$ 
I \left( 1 + \frac{4\pi^2}{\left(\ln r \right)^2} \right) = \frac{4\pi}{\left(\ln r\right)^2} \left( 1+ r^{\frac{1}{2}}\right)  
$$
so that
$$\eps^{\frac{1}{2}} \sqrt{ \ln\frac{1}{r}} \int_0^1 r^{\frac{v}{2}}f(v) \sqrt{1+\eps^2 r^{2v}}dv = \eps^{\frac{1}{2}}\sqrt{ \ln\frac{1}{r}} I (1+O(\eps^2)) = O\left( \eps^{\frac{1}{2}}\left(\ln\frac{1}{r}\right)^{-\frac{3}{2}} \right)$$
as $\eps \to 0$. 
Thanks to  \eqref{eqthetasquare} \eqref{eqthetamean} and \eqref{eqgradtheta} we get from the variational characterization that
$$ 
\sigma_\eps^1  \leq \frac{\int_{\Omega} \left\vert \nabla\phi \right\vert^2_{g_{\eps}} dA_{\eps}  }{\int_{I^{+} \cup I^{-}} \phi^2 dl_{\eps} - \frac{\left(\int_{I^{+} \cup I^{-}} \phi dl_{\eps}\right)^2}{1-2\pi\eps^2 + \frac{2\eps}{t} +O(\eps^3)}}   \leq t_{\eps}  \frac{ \int_0^1   \frac{f^2}{4}  dv + \int_0^1 \frac{(f')^2}{\left(\ln\frac{1}{r}\right)^2} dv }{ 2 \int_0^1f^2  dv - O\left(\eps\left(\ln\frac{1}{r}\right)^{-3} + \eps^2 \right) } 
$$
as $\eps \to 0$, which gives the inequality \eqref{eqineqsigmaeps}.

Now, in order to prove the inequality \eqref{ineqsigmaeps2} on $\sigma_{\eps}^{K+1}$, it suffices to take the previous test functions (note that we have $K$ linearly independent functions of the first type) and notice that they are orthogonal up to a small error term we shall compute:
\begin{eqnarray*}
 \int_{\Omega_{\eps}}\left\langle \nabla u_{\eps},\nabla \phi \right\rangle_{g_{\eps}} dA_{g_{\eps}} & = & \eps \int_{r}^1 y^2  u_{\eps,y} \phi_{y} \sqrt{1+\eps^2 y^2}dy  \\
 & = & \frac{t_{\eps}^{\frac{1}{2}}\eps^{\frac{1}{2}}\left(u_{\star}(p_0) - u_{\star}(p_1)\right)}{\left(\ln \frac{1}{r} \right)^{\frac{1}{2}}} \int_{0}^1 r^{\frac{v}{2}} \left( \frac{f'}{2} + \frac{f}{\ln\frac{1}{r}} \right) \sqrt{1+\eps^2 r^{2v}} dv  \\
 & = & O\left( \eps^{\frac{1}{2}} \left(\ln \frac{1}{r}\right)^{-\frac{5}{2}}  \right) +O(\eps^2)
\end{eqnarray*} 
where we use the computation of $I = \int_{0}^1 r^{\frac{v}{2}} \sin(\pi v)dv $ and we also compute $\int_{0}^1 r^{\frac{v}{2}} \cos(\pi v)dv $ by integration by parts. We also have that
\begin{eqnarray*}
 \int_{I^+ \cup I^-} u_{\eps}\phi  dl_{g_{\eps}} & = & \frac{\eps}{t_{\eps}} \int_{r}^1 u_{\eps} \phi \sqrt{1+\eps^2 y^2}dy  \\
 & = & \frac{\eps^{\frac{1}{2}}\left(\ln \frac{1}{r} \right)^{\frac{1}{2}}}{t_{\eps}^{\frac{1}{2}}} \int_{0}^1 r^{\frac{v}{2}} ((u_{\star}(p_1)-u_{\star}(p_0) )v+ u_{\star}(p_0)) f(v) \sqrt{1+\eps^2 r^{2v}} dv  \\
 & = & O\left( u_{\star}(p_0) \eps^{\frac{1}{2}} \left(\ln \frac{1}{r}\right)^{-\frac{3}{2}} + \eps^{\frac{1}{2}} \left(\ln \frac{1}{r}\right)^{-\frac{5}{2}} + \eps^2 \right)\hskip.1cm,
\end{eqnarray*} 
where we can prove by several integrations by parts that $\int_{0}^1 r^{\frac{v}{2}} v \sin(\pi v)dv = O\left( \left(\ln \frac{1}{r}\right)^{-3}\right)$.

Moreover, if $u_\star$ and $v_\star$ are two orthonormal $\sigma_\star$-eigenfunctions on $\Sigma$, we have that
\begin{align*}
\int_{\Sigma_\eps} | \langle \nabla u_\eps , \nabla v_\eps \rangle_{g_\eps} | dA_\eps
\leq&
\int_\Sigma |\nabla u_\star| |\nabla (v_\eps-v_\star)| 
+ \int_\Sigma |\nabla v_\star| |\nabla(u_\eps-u_\star)| 
\\
&+\int_\Sigma |\nabla(u_\eps-u_\star)| |\nabla (v_\eps-v_\star)| 
+\eps \int_r^1 y^2 u_{\eps,y} v_{\eps,y} dy
\\
\leq&
C \eps^4
+
C \frac{\eps(u_\star(p_0)-u_\star(p_1))(v_\star(p_0)-v_\star(p_1))}{\left( \ln \frac{1}{r}\right)^2}
\\
\leq& C \eps^4 + \frac{C \eps}{\left(\ln \frac{1}{r}\right)^2}.
\end{align*}
Similarly, assuming that $u_\star(p_0)=0$, we also have that
\begin{align*}
\int_{\partial \Sigma_\eps} |u_\eps v_\eps| dl_\eps
&\leq
C \eps^2 +C \frac{\eps}{t_\eps} \int_r^1 |u_\eps  v_\eps| dy
\leq C \eps^4 + C \frac{\eps}{t_\eps} \int_r^1 \frac{\ln y}{\ln r} dy
\\
&\leq
C \eps^4 + C \frac{\eps}{\ln \left(\frac{1}{r}\right)}.
\end{align*}

Finally let us choose $u_\star^1,\dots, u_\star^K$ an orthonormal basis of $\sigma_\star$-eigenfunctions such that
$u_\star^i(p_0)=0$ for $i \geq 2$ and denote by $u_\eps^1,\dots, u_\eps^K$ the corresponding extensions constructed above.
The estimate \eqref{ineqsigmaeps2} now easily follows from the estimates above and the variational characterization of eigenvalues applied to the space spanned by $u_\eps^1,\dots, u_\eps^K,\phi$ from above.
\end{proof}

\section{Pointwise estimates on eigenfunctions} \label{section4}

We first aim at giving pointwise estimates on eigenfunctions at the neighbourhood of $p_0$ and $p_1$ on $\Sigma$ in order to get
some control on the boundary values of the eigenvalue equation on $\Omega_\eps$.
More precisely, we want to get estimates on $\overline{\theta}(0)$ and $\overline{\theta}(1)$, where we use the change of variables to the function $\theta$ as described in \cref{sec_ansatz}.
It is natural to compare them to $u_{\star}(p_0)$ and $u_{\star}(p_1)$, where $u_{\star}$ is the weak limit in $W^{1,2}(\Sigma)$ (and strong limit in $L^2(\partial \Sigma)$) of $u_{\eps}$. Notice that by the elliptic estimates of \cite{robin} on domains with corners, the functions $u_{\eps}$ are $C^{0,\alpha}$ up to the boundary. We have the following compatibility conditions:
\begin{equation} \label{eqcompatibilityp0} 
\frac{1}{\eps^2}\int_{-\frac{\eps^2}{2}}^{\frac{\eps^2}{2}} \left( u_{\eps}\circ \varphi_0^{-1} \right) (x,0)dx = \overline{\phi}(1) = \frac{\sqrt{t_{\eps}}\overline{\theta}(0)}{\sqrt{\eps}\sqrt{\ln\frac{1}{r}}} \hskip.1cm,
\end{equation}
\begin{equation} \label{eqcompatibilityp1} 
\frac{1}{r^2\eps^2}\int_{-\frac{r^2\eps^2}{2}}^{\frac{r^2\eps^2}{2}} \left( u_{\eps}\circ \varphi_1^{-1} \right) (x,0)dx  = \overline{\phi}(r) = \frac{\sqrt{t_{\eps}}r^{-\frac{1}{2}} \overline{\theta}(1)}{\sqrt{\eps}\sqrt{\ln\frac{1}{r}}} \hskip.1cm.
\end{equation}
Because of the factor $r^{-\frac{1}{2}}$ in \eqref{eqcompatibilityp1} that is not present in \eqref{eqcompatibilityp0}, we see that $\overline{\theta}(0)$ and $\overline{\theta}(1)$ do not play the same role. 
More precisely, $\overline{\theta}(1)$ will be much smaller. 
Therefore, we will never need a very precise estimate in the neighborhood of $p_1$. 

\begin{claim} \label{claimpointwise}
Let $u_\eps$ be an $L^2(\partial \Sigma_\eps)$-normalized $\sigma_\eps^l$-eigenfunction for $l \in \{1,\dots,K+1\}$.
We have a constant $C>0$ independent of $\eps$ and $r$ such that
\begin{equation} \label{eqestneibpi2} 
\left\vert \left( u_{\eps}\circ \varphi_1^{-1} \right) (x) \right\vert \leq C \ln\frac{1}{r \eps}  
\end{equation}
for any $x \in \mathbb{D}_{r^2 \eps^2}^+$ up to the boundary, and
\begin{equation} \label{eqestneibpi} 
\left\vert \left( u_{\eps}\circ \varphi_0^{-1} \right) (x)  \right\vert \leq C \ln\frac{1}{\eps}   
\end{equation}
for any $ x \in \mathbb{D}_{\eps^2}^+$ up to the boundary. 
More precisely,
\begin{equation} \label{eqestneibp0precise} 
\left\vert \left( u_{\eps}\circ \varphi_0^{-1} \right) (x) - u_{\star}(p_0)  \right\vert \leq C \left(\left\| u_{\eps} - u_{\star} \right\|_{W^{1,2}(\Sigma)} + \left\| \nabla \phi_{\eps} \right\|_{L^{2}(F_{\eps})} + \left\vert \sigma_{\eps} - \sigma_{\star}  \right\vert+ \eps + b_{\eps}\ln \frac{1}{\eps} \right)   
\end{equation}
for any $ x \in \mathbb{D}_{\eps^2}^+$ up the boundary, where 
$F_{\eps} = \{(x,y)\in \Omega ; 1-\eps \leq y\leq 1\}$ and
\begin{equation}\label{eqdefbepsi} 
b_{\eps} = \frac{\eps}{\pi}\overline{\phi_y}(1)\hskip.1cm. 
\end{equation}
\end{claim} 

\begin{proof}
We set 
\begin{equation} 
m^i_{\eps}(\rho) = \frac{1}{\pi} \int_{0}^{\pi} \left( u_{\eps}\circ \varphi_i^{-1} \right) \left(\rho \left( \cos \theta , \sin \theta \right)  \right)  d\theta \
\end{equation}
for $i=0,1$, the mean value of $u_{\eps}$ on the arc of radius $\rho$ in the neighbourhood of $p_i$ in the conformal chart $\varphi_i$. 
Since the charts $\varphi_i$ are conformal and thanks to the eigenvalue equation on $\partial\Sigma\setminus \left(A_0 \cup A_1\right)$, we know that
\begin{equation} \label{eqsteklovueps} 
\begin{cases}
 \Delta \left(u_{\eps}\circ\varphi_i^{-1}\right) = 0 &  \text{in} \ \mathbb{D}_2^+
\\
-\partial_{y} \left(u_{\eps}\circ\varphi_0^{-1}\right) = e^{\omega_0} \sigma_{\eps} \left( u_{\eps}\circ\varphi_0^{-1} \right) &  \text{on} \ [-2,2]\times\{0\} \setminus \left([-\frac{\eps^2}{2},\frac{\eps^2}{2}]\times\{0\}\right) \text{ if } i=0 \\
-\partial_{y} \left(u_{\eps}\circ\varphi_1^{-1}\right) = e^{\omega_1} \sigma_{\eps}\left(  u_{\eps}\circ\varphi_1^{-1} \right) &  \text{on} \  [-2,2]\times\{0\} \setminus \left([-\frac{r^2\eps^2}{2},\frac{r^2\eps^2}{2}]\times\{0\}\right) \text{ if } i=1  \\
\end{cases}
 \end{equation}
so that
$$
 - \frac{1}{\rho} \partial_\rho \left(\rho \left(m_\eps^i \right)' \right) = \frac{e^{\omega_i} \sigma_{\eps}}{\pi \rho} \left(u_{\eps}\circ\varphi_i^{-1}(\rho,0) + u_{\eps}\circ\varphi_i^{-1}(-\rho,0) \right)
 $$
for any $ \frac{\eps^2}{2} <\rho \leq 1$ if $i = 0$ and for any $ \frac{r^2\eps^2}{2} <\rho \leq 1$ if $i=1$. 
We integrate this equation to obtain
\begin{equation} \label{eq_green_integrated}
\rho \left(m_{\eps}^i\right)'(\rho) = \left(m_{\eps}^i\right)'(1) -  \frac{\sigma_{\eps}}{\pi} \int_{J_1 \setminus J_\rho}u_{\eps}dl_g 
\end{equation}
where $J_s = \varphi_i^{-1}\left(\left[ -s,s\right]\times \{0\}\right)$. 
Integrating again, we get that
$$ 
m_{\eps}^i(\rho) = m_{\eps}^i(1) + \ln\rho \left(m_{\eps}^i\right)'(1) - \frac{\sigma_{\eps}}{\pi} \int_{1}^{\rho}\frac{1}{s}\left(\int_{J_1 \setminus J_s}u_{\eps}dl_g\right)ds.
$$
Moreover, we have by H{\"o}lder's inequality
that
$$
\left| \int_1^\rho \frac{1}{s} \left( \int_{J_1\setminus J_s} u_\eps dl_\eps \right) ds \right|
\leq
 \int_1^\rho \frac{1}{s} \int_{J_1\setminus J_\rho} |u_\eps| dl_\eps ds 
 \leq
 \left( \ln \frac{1}{\rho} \right) \sqrt{N} \sqrt{L_g(J_1 \setminus J_\rho)},
$$
which then implies by the estimates above that
$$
 \left\vert m_{\eps}^i(\rho) \right\vert \leq \left\vert m_{\eps}^i(1)\right\vert +  \left\vert\left(m_{\eps}^i\right)'(1)\right\vert \left(\ln\frac{1}{\rho}\right) + \frac{\sigma_{\eps}}{\pi}  \sqrt{N}\sqrt{L_g(J_1\setminus J_\rho)} \left(\ln\frac{1}{\rho}\right)\hskip.1cm.
 $$
By standard elliptic theory on $\eqref{eqsteklovueps}$, we know that
$$ \left\vert m_{\eps}^i(1)\right\vert +  \left\vert\left(m_{\eps}^i\right)'(1)\right\vert \leq C \sqrt{N} $$
as $\eps \to 0$ so that we obtain 
\begin{equation}
\left\vert m_{\eps}^i(\rho) \right\vert \leq C\sqrt{N} \left(1 + \ln\frac{1}{\rho}\right)
\end{equation}
for $\rho \geq \eps^2$ if $i=0$ and for $\rho \geq r^2\eps^2$ if $i=1$. 

Let us be more precise for $i=0$. 
We set 
$$
b_{\eps} = \frac{\eps^2}{2} \left(m_{\eps}^0\right)'\left(\frac{\eps^2}{2}\right).
$$ 
Notice that by the computation above (more precisely \eqref{eq_green_integrated}) and a similar application of Green's formula, we have that
\begin{equation}\label{eqdefbepsi2} 
b_{\eps} = \left(m^0_{\eps}\right)'(1) - \frac{\sigma_{\eps}}{\pi}\int_{J_1\setminus J_{\frac{\eps^2}{2}}} u_{\eps} dl_g = \frac{\eps}{\pi}\overline{\phi_y}(1)\hskip.1cm, 
\end{equation}
so that in particular our notation for $b_\eps$ is consistent with \eqref{eqdefbepsi}. 
By integrating \eqref{eqdefbepsi2} as before we then deduce the following formula
\begin{align*} 
m_{\eps}^0(\rho) 
&=
m_\eps^0(1)+\ln \rho\, \left(m_\eps^0\right)'(1) - \frac{\sigma_\eps}{\pi} \int_1^\rho \frac{1}{s} \left( \int_{J_1\setminus J_s} u_\eps dl_g \right) ds
\\
&=
m_\eps^0(1) + b_{\eps} \ln \rho + \ln \rho \frac{\sigma_{\eps}}{\pi}\int_{J_1\setminus J_{\frac{\eps^2}{2}}} u_{\eps} dl_g
 - \frac{\sigma_\eps}{\pi} \int_1^\rho \frac{1}{s} \left( \int_{J_1\setminus J_s} u_\eps dl_g \right) ds
 \\
&=
m_\eps^0(1) + b_{\eps} \ln \rho + \ln \rho \frac{\sigma_{\eps}}{\pi}\int_{J_\rho \setminus J_{\frac{\eps^2}{2}}} u_{\eps} dl_g
+ \frac{\sigma_\eps}{\pi} \int_1^\rho \frac{1}{s} \left( \int_{J_s \setminus J_\rho} u_\eps dl_g \right) ds
\\
&=
m_{\eps}^0(1) + b_{\eps} \ln \rho  + \ln \rho \frac{\sigma_{\eps}}{\pi} \int_{J_{\rho}\setminus J_{\frac{\eps^2}{2}}} u^i_{\eps} dl_g + \frac{\sigma_\eps}{\pi} \int_1^\rho \frac{1}{s} \left( \int_{J_s \setminus J_\rho} u_\eps dl_g \right) ds
 \hskip.1cm.
 \end{align*}
We can do the very same computation for $m_{\star}^0$, the mean value of $u_{\star}$ on the arc of radius $\rho$ in the neighbourhood of $p_0$ in the conformal chart $\varphi_0$, where $u_{\star}\circ \varphi_0^{-1}$ satisfies the same equation with eigenvalue $\sigma_{\star}$ but now also along $J_{\frac{\eps^2}{2}}$ (which does not hold for $u_\eps$).
This gives that
$$ 
m_{\star}^0(\rho) 
= 
m_{\star}^0(1)   +  \ln \rho \frac{\sigma_{\star}}{\pi} \int_{J_{\rho}} u_{\star} dl_g 
+ \frac{\sigma_{\star}}{\pi}  \int_{1}^{\rho}  \frac{1}{s} \left( \int_{J_s \setminus J_\rho} u_\star dl_g \right) ds \hskip.1cm.$$
The difference between $m_{\eps}^0$ and $m_{\star}^0$ gives
\begin{eqnarray*}
m_{\eps}^0(\rho) & = & m_{\star}^0(\rho) + b_{\eps} \ln \rho + \left(m_{\eps}^0 -  m_{\star}^0\right)(1) \\
&  & 
+ \ln \rho \left( \frac{\sigma_{\eps}}{\pi} \int_{J_{\rho}\setminus J_{\frac{\eps^2}{2}}} \left(u_{\eps} - u_{\star}\right) dl_g + \frac{\sigma_{\eps}-\sigma_{\star}}{\pi} \int_{J_{\rho}\setminus J_{\frac{\eps^2}{2}}} u_{\star} dl_g 
-\frac{\sigma_{\star}}{\pi} \int_{ J_{\frac{\eps^2}{2}}} u_{\star} dl_g \right) \\
& & 
+
\frac{\sigma_\eps}{\pi} \int_1^\rho \frac{1}{s} \left( \int_{J_s \setminus J_\rho} (u_\eps - u_\star) dl_g \right) ds
 + \frac{\sigma_{\eps}-\sigma_{\star}}{\pi} \int_1^\rho \frac{1}{s} \left( \int_{J_s \setminus J_\rho} u_\star dl_g \right) ds.
\end{eqnarray*}
Moreover, we can estimate
\begin{align*}
\left| \int_1^\rho \frac{1}{s} \left( \int_{J_s \setminus J_\rho} (u_\eps - u_\star) dl_g \right) ds \right|
&\leq
\int_1^\rho \frac{1}{s} \sqrt{L_g(J_s \setminus J_\rho)} \|u_{\eps} - u_\star\|_{L^2(J_\rho \setminus J_s)} ds
\\
& \leq C \|u_{\eps} - u_\star\|_{L^2(\partial \Sigma)} \int_1^\rho \frac{(s-\rho)^{1/2}}{s} ds 
\\
&\leq 
C \|u_{\eps} - u_\star\|_{L^2(\partial \Sigma)}
\end{align*}
and similarly for the last term in the difference of the mean values above.
Therefore, we get by standard elliptic estimates on a compact subset of $\Sigma\setminus\{p_0,p_1\}$ and H\"older's inequality that there is a constant $C$ independent of $\eps$ and $\rho$ such that
\begin{equation} \label{ineqmeanvaluemeps} 
\left\vert m_{\eps}^0(\rho) - m_{\star}^0(\rho) - b_{\eps}\ln\rho \right\vert \leq C \left(\left\| u_{\eps} - u_{\star} \right\|_{L^{2}(\partial\Sigma)} + \left\vert \sigma_{\eps} - \sigma_{\star}  \right\vert+ \eps^2 \ln \frac{1}{\rho}\right) 
\end{equation}
for any $\rho \in \mathbb{D} \setminus \mathbb{D}_{\frac{\eps^2}{2}}$.

Now we look at $u_{\eps}$ in the charts $f_i$ and $g_i$ that we used to define the gluing between $\Omega_{\eps}$ and $\Sigma$ at the neighborhood of $p_i$ for $i=0,1$. We denote by $v_{\eps}^i$ the functions in these charts and the equation \eqref{eqsteklovueps} becomes
\begin{equation} \label{eqsteklovueps2} \begin{cases}
 \Delta v_{\eps}^0 = 0 &  \text{in} \ \mathbb{D}_\frac{2}{\eps^2}^+ \cup E_0
\\
-\partial_{y} v_{\eps}^0 = \eps^2 e^{\omega_0(\eps^2 (x,z))} \sigma_{\eps} v^0_{\eps} &  \text{in} \ \left[-\frac{2}{\eps^2},\frac{2}{\eps^2}\right]\times\{0\} \setminus \left( \left[-\frac{1}{2},\frac{1}{2}\right]\times\{0\}\right) \\
\partial_{\nu_0^{\pm}} v_{\eps}^0 = \frac{\eps^2}{t_{\eps}} \sigma_{\eps} v_{\eps}^0 &  \text{if} \ \frac{r-1}{\eps} \leq z\leq 0 \text{ and }   x = \pm \frac{\left(1+\eps z\right)^2}{2}  \\
\Delta v_{\eps}^1 = 0 &  \text{in} \ \mathbb{D}_\frac{2}{r^2\eps^2}^+ \cup E_1
\\
-\partial_{y} v_{\eps}^1 = r^2\eps^2 e^{\omega_1(r^2\eps^2 (x,z))} \sigma_{\eps} v_{\eps}^1 &  \text{in} \ \left[-\frac{2}{r^2\eps^2},\frac{2}{r^2\eps^2}\right]\times\{0\} \setminus \left( \left[-\frac{1}{2},\frac{1}{2}\right]\times\{0\}\right) \\
\partial_{\nu_1^{\pm}} v_{\eps}^1 = \frac{r^2\eps^2}{t_{\eps}} \sigma_{\eps} v_{\eps}^1 &  \text{if} \  \frac{r-1}{r^{2}\eps} \leq z\leq 0 \text{ and }   x = \pm \frac{\left(1-r\eps z\right)^2}{2}  \\
\end{cases} \end{equation}
where $E_i$ for $i=0,1$, defined by \eqref{eqdefE0} and \eqref{eqdefE1}, are the image of $\Omega_{\eps}$ under the charts $g_i$, and $\nu_i^{\pm}$ is the outpointing normal on the boundaries of $E_i$ endowed with the flat metric. By elliptic regularity on domains with corners \cite{robin}, we know that $v_{\eps}^i \in \mathcal{C}^{0,\alpha}$ and we have the estimates
$$  
\left\| v_{\eps}^0 -  m_{\eps}^0(\eps^2) \right\|_{\mathcal{C}^{k}(F_0)} \leq C\left( \eps + 1 \right)
$$
$$ 
\left\| v_{\eps}^1 -  m_{\eps}^1(r^2 \eps^2) \right\|_{\mathcal{C}^{k}(F_1)} \leq C\left( r \eps + 1 \right)
$$
for some constant $C$ as soon as $F_i$ is a bounded set at the neighborhood of $(0,0)$. Therefore, \eqref{eqestneibpi} and \eqref{eqestneibpi2} hold true. The energy which appears in the right-hand term of the pointwise estimate for $i=0$ gives the precise estimate \eqref{eqestneibp0precise}.
\end{proof}

\section{Asymptotic expansion on the first eigenvalue and first eigenfunction} \label{section5}

In this section we prove our main technical tool, a precise asymptotic expansion of an eigenfunction on the cuspidal domain with control on boundary values given by \cref{claimpointwise}.

\subsection{Preliminary computations}

Let $u_{\eps}$ be an eigenfunction associated to $\sigma^1_{\eps}$ with unit norm, that is $\int_{\partial\Sigma_{\eps}} u_{\eps}^2 dl_{\eps} = 1$. 
Integrating the equation satisfied by $u_{\eps}$, we get that
\begin{equation} \label{eqenergyeigenfunc} 
\int_{\Sigma} \left\vert \nabla u_{\eps} \right\vert^2_g dA_g - \sigma_{\eps} N + \int_{\Omega_{\eps}} \left\vert \nabla \phi_{\eps} \right\vert^2_{g_{\eps}} dA_{\eps} =  \sigma_{\eps} M \hskip.1cm, 
\end{equation}
where we write $\phi_{\eps}=\left. u_\eps \right|_{\Omega_\eps}$ for the eigenfunction $u_{\eps}$ in the chart $\Omega_{\eps}$ of $\Sigma_{\eps}$, 
$$
M = \int_{I^+ \cup I^-} \phi_{\eps}^2 dl_{\eps}   \text{ and }  N = \int_{\partial\Sigma\setminus \left(A_0 \cup A_1\right)} u_{\eps}^2 dl_{g}
$$
are the boundary masses of the eigenfunctions on the cuspidal domain and on the surface where 
$$
A_0 = \varphi_0^{-1}\left(\mathbb{D}_{\eps^2}^+ \right) \cap \partial{\Sigma}
\hspace{0.2cm} \text{and} \hspace{0.2cm}
A_1 = \varphi_1^{-1}\left(\mathbb{D}_{r^2\eps^2}^+ \right) \cap \partial\Sigma.
$$ 
Notice that by assumption $M+N = 1$. 
We define a new function $\theta$ by the change of variables \eqref{eqdeftheta}, and thanks to \eqref{eqgradtheta} we can rewrite the gradient term over
$\Omega_\eps$ so that
 \eqref{eqenergyeigenfunc} becomes 
\begin{equation} \label{eqenergyeigenfunc2} 
\frac{\delta_{\eps}}{t_{\eps}} + \frac{1}{\eps^2}  \left\| \theta_x \right\|_{L^2(\widetilde{\Omega})}^2   +  \int_0^1 \overline{ \left( \frac{\theta}{2} + \frac{\theta_v}{\ln \frac{1}{r}} \right)^2 } dv =  \frac{\sigma_{\eps}}{t_{\eps}} M \hskip.1cm, 
\end{equation}
where
\begin{equation} \label{eqdefdeltaeigen} 
\delta_{\eps} = \int_{\Sigma} \left\vert \nabla u_{\eps} \right\vert^2_g dA_g - \sigma_{\eps} N 
\end{equation}
and we have that
$$  
\int_0^1 \overline{ \left( \frac{\theta}{2} + \frac{\theta_v}{\ln \frac{1}{r}} \right)^2 } dv = \int_0^1 \left(\frac{\overline{\theta^2}}{4} + \frac{\overline{\theta\theta_v}}{\ln \frac{1}{r} } +  \frac{\overline{\theta_v^2}}{\left( \ln\frac{1}{r} \right)^2} \right)dv \hskip.1cm,
$$
so that \eqref{eqenergyeigenfunc2} becomes
\begin{equation} \label{eqenergyeigenfunc3} 
M \left(\frac{\sigma_{\eps}}{t_{\eps}} - \frac{1}{8}\right)(1+O(\eps^2)) =  \frac{\delta_{\eps}}{t_{\eps}} + \frac{\left\| \theta_x \right\|_{L^2(\widetilde{\Omega})}^2}{\eps^2}  + \frac{1}{\left(\ln r\right)^2}\int_0^1 \overline{\theta_v^2} dv + I_1 + I_2  \hskip.1cm,
\end{equation}
where 
$$ 
I_1 = \frac{1}{8} \int_0^1 \left( 2\overline{\theta^2} - \theta^2\left(\frac{r^{2v}}{2},v\right) - \theta^2\left(-\frac{r^{2v}}{2},v\right) \right) dv
$$
and 
$$ I_2 = \frac{1}{\ln \frac{1}{r}}   \int_0^1 \overline{\theta\theta_v} dv \hskip.1cm.$$
We have that 
$$ 
\overline{\theta\theta_v} = \overline{\left(\theta - \overline{\theta}\right)\theta_v} +  \overline{\bar{\theta} \theta_v} = \overline{\left(\theta - \overline{\theta}\right)\theta_v} + \overline{\theta} \cdot  \overline{ \theta_v} \hskip.1cm.
$$
But since $\left(\overline{\theta}\right)'= \overline{ \theta_v} + \ln\left(\frac{1}{r}\right) \left( 2\overline{\theta} - \theta(\frac{r^{2v}}{2},v) - \theta(-\frac{r^{2v}}{2},v)\right) $ 
and 
$$
2 \int_0^1 \overline{\theta}\left(\overline{\theta}\right)'  dv = \overline{\theta}^2(1) - \overline{\theta}^2(0) \hskip.1cm,
$$ 
we get 
$$
I_2 = \frac{1}{2\ln \frac{1}{r}}\left(\overline{\theta}^2(1) - \overline{\theta}^2(0)\right) + I_3 \hskip.1cm,
$$ 
where
$$ 
I_3 = \frac{1}{\ln\frac{1}{r}} \int_0^1 \overline{\left(\theta - \overline{\theta}\right)\theta_v}dv -   \int_0^1 \overline{\theta} \left( 2\overline{\theta} - \theta\left(\frac{r^{2v}}{2},v\right) - \theta\left(-\frac{r^{2v}}{2},v\right)\right)dv \hskip.1cm.
$$
Let $-\frac{r^{2v}}{2} \leq x \leq \frac{r^{2v}}{2} $, then
\begin{equation} \label{eqestthetalinex} 
\left\vert \theta(x,v) - \bar\theta(v) \right\vert 
=  
r^{-2v}  \left\vert  \int_{-\frac{r^{2v}}{2}}^{\frac{r^{2v}}{2}}\left(\int_{x}^{s} \theta_t(t,v)dt\right) ds \right\vert  
\leq 
\left(r^{4v}\overline{\theta_x^2}(v)\right)^{\frac{1}{2}} \hskip.1cm. 
\end{equation}
Note that since $r^{2v} \leq 1$ this implies that we also have that
\begin{align} \label{eq_bar_theta_l2}
\| \bar \theta \|_{L^2(0,1)} \leq \left\| \bar \theta - \theta \left( \frac{r^{2v}}{2},v\right) \right\|_{L^2(0,1)} + \sqrt{M}
\leq \| \theta_x \|_{L^2(\tilde \Omega)} + \sqrt{M}.
\end{align}
Similarly, discarding $r^{2v} \leq 1$ again, thanks to H{\"o}lder's inequality,
\begin{align*}
\left\vert  I_1 \right\vert 
&\leq 
\frac{1}{8} \left( \left\| \bar\theta + \theta\left(\frac{r^{2v}}{2},v\right) \right\|_{L^2(0,1)} + \left\| \bar\theta + \theta\left(\frac{-r^{2v}}{2},v\right) \right\|_{L^2(0,1)} \right)  \left\| \theta_x \right\|_{L^2(\widetilde{\Omega})}
\\
& \leq
\frac{1}{2} ( \left\| \theta_x \right\|_{L^2(\widetilde{\Omega})} + \sqrt{M} ) \left\| \theta_x \right\|_{L^2(\widetilde{\Omega})},
\end{align*}
and
\begin{align*}
\int_0^1 \left| \overline{(\theta-\overline{\theta}) \theta_v} \right| \, dv
& \leq
\int_0^1 \int_{-\frac{r^{2v}}{2}}^{\frac{r^{2v}}{2}} \overline{\theta_x^2}^{\frac{1}{2}} \theta_v \, dx dv
\\
&\leq
\left(\int_0^1 \int_{-\frac{r^{2v}}{2}}^{\frac{r^{2v}}{2}} \overline{\theta_x^2} \, dx dv \right)^{1/2}
\left(\int_0^1 \int_{-\frac{r^{2v}}{2}}^{\frac{r^{2v}}{2}} \theta_v^2 \, dx dv \right)^{1/2}
\\
& \leq
\| \theta_x \|_{L^2(\tilde \Omega)} \left(\int_0^1 r^{2v}\overline {\theta_v^2} \, dv \right)^{1/2}.
\end{align*}
After discarding $r^{2v} \leq 1$ in the line above we then find from the above estimates that
$$
 \left\vert I_3 \right\vert 
 \leq 
 \frac{1}{\ln\frac{1}{r}} \left(\int_0^1 \overline{\theta_v^2}dv\right)^{1/2} \left\| \theta_x \right\|_{L^2(\widetilde{\Omega})} +   2 \left\| \bar\theta\right\|_{L^2(0,1)} \left\| \theta_x \right\|_{L^2(\widetilde{\Omega})} \hskip.1cm.$$
Gathering our estimates, using again that $\left\| \theta_x \right\|_{L^2(\widetilde{\Omega})} \leq \eps$ by \eqref{eqenergyeigenfunc3},
we then have thanks to Young's inequality that
\begin{equation} \label{eqestimateerrore} 
e_{\eps} = \left\vert  I_1 \right\vert + \left\vert  I_3 \right\vert 
= 
O\left( \frac{ \left\| \theta_x \right\|_{L^2(\widetilde{\Omega})}^2}{\eta_{\eps}} +  \eta_{\eps}\frac{\left\|\overline{\theta_v^2} \right\|_{L^1(0,1)}}{\left(\ln r\right)^2} + \eta_{\eps}  M   \right) \hskip.1cm
\end{equation}
 for any $\eta_{\eps} \in (0,1]$.
Now, we can rewrite \eqref{eqenergyeigenfunc3} as follows
\begin{equation} \label{eqenergyeigenfunc4}
M \left(\frac{\sigma_{\eps}}{t_{\eps}} - \frac{1}{8}\right) =  \frac{\delta_{\eps}}{t_{\eps}}  +  \frac{\left\| \theta_x \right\|_{L^2(\widetilde{\Omega})}^2}{\eps^2} +\frac{1}{\left(\ln r\right)^2}\int_0^1 \overline{\theta_v^2} dv + \frac{1}{2 \ln \frac{1}{r}}\left(\overline{\theta}^2(1) - \overline{\theta}^2(0)\right) + O(e_{\eps} + \eps^2 M) \hskip.1cm. 
\end{equation}
The formula \eqref{eqenergyeigenfunc4} gives the main connection between the behaviour of the eigenfunction $u_{\eps}$ on the thick part $\Sigma$ and $\Omega_{\eps}$ the thin part.

\bigskip

We now incorporate our pointwise estimate \cref{claimpointwise} into \eqref{eqenergyeigenfunc4}.
Notice that by \eqref{eqestneibpi2} at the neighbourhood of $p_1$ and by \eqref{eqcompatibilityp1}, we have that
$$
\frac{\overline{\theta}^2(1)}{2\ln\frac{1}{r}} =  \frac{r \eps}{t_{\eps}} \frac{\overline{\phi}(1)^2}{2} = O\left(r\eps\left(\ln \frac{1}{r\eps}\right)^2\right) 
$$
and  \eqref{eqenergyeigenfunc4} becomes
\begin{equation} \label{eqenergyeigenfunc5}
M \left(\frac{\sigma_{\eps}}{t_{\eps}} - \frac{1}{8}\right) =  \frac{\delta_{\eps}}{t_{\eps}}  +  \frac{\left\| \theta_x \right\|_{L^2(\widetilde{\Omega})}^2}{\eps^2} +\frac{1}{\left(\ln r\right)^2}\int_0^1 \overline{\theta_v^2} dv - \frac{\overline{\theta}^2(0)}{2 \ln \frac{1}{r}} + O\left(e_{\eps}+ \eps^2 M + r\eps\left(\ln \frac{1}{r\eps}\right)^2 \right) \hskip.1cm. 
\end{equation}
Since $L_g(\partial \Sigma)=1$ by the Poincar\'e trace inequality, we have that
$$
\sigma_\star \left(N - \left(\int_{\partial \Sigma}  u_\eps dl_g \right)^2  \right)
\leq
\sigma_\star \left( \int_{\partial \Sigma}  u_\eps^2 dl_g - \left(\int_{\partial \Sigma}  u_\eps dl_g \right)^2  \right)
\leq \int_\Sigma |\nabla u_\eps|^2 dA_g,
$$
where we recall that $\sigma_{\star}$ is the first non-zero Steklov eigenvalue on $\Sigma$. 
This immediately implies that $\delta_{\eps}$ defined by \eqref{eqdefdeltaeigen} can be estimated according to
\begin{equation} \label{eqgapsigmastarsigmaeps}
N \left( \sigma_{\star} - \sigma_{\eps} \right) \leq \delta_{\eps} + \sigma_{\star} \left(\int_{\partial\Sigma} u_{\eps} dA_g\right)^2.
\end{equation}
Since $u_{\eps}$ has zero mean value on $\partial \Sigma_{\eps}$, we have that
$$ 
\left(\int_{\partial\Sigma \setminus \left(A_0 \cup A_1\right)} u_{\eps} dl_g\right)^2 =  \left(\int_{I^+ \cup I^-} \phi_{\eps} dl_{\eps}\right)^2.
$$ 
Moreover, we have thanks to \cref{claimpointwise} that
$$
\int_{ A_0 \cup A_1} u_{\eps} dl_g = O\left(\eps^2 \ln \eps \right)
$$
so that we find
\begin{equation} \label{eqgapsigmastarsigmaeps2}  
\left( \sigma_{\star} - \sigma_{\eps} \right) \leq \frac{\delta_{\eps}}{N} + 2\frac{\sigma_{\star}}{N}\left(  \left( \int_{I^+ \cup I^-} \phi_{\eps} dl_{\eps}  \right)^2 + O\left(\eps^4 \left(\ln\eps\right)^2\right) \right)  \hskip.1cm,
\end{equation}
where we also remark that
$$
\left( \int_{I^+ \cup I^-} \phi_{\eps} dl_{\eps}  \right)^2 = O(M \eps)
$$
thanks to H{\"o}lder's inequality.

The first idea is then to make $\delta_{\eps}$ as small as possible in order to have the expected inequality.

\subsection{Choice of the parameter $t_\eps$}

We now aim at choosing the adapted parameters $t_{\eps}$ (near $t_\star$) and $r_{\eps}$ for which we we have a chance to minimize $\delta_{\eps}$.
We choose 
$$
r_{\eps} = \exp\left(-\frac{1}{\eps^{\alpha}}\right)
$$
for $0<\alpha<\frac{1}{2}$ and.
Let us also introduce the parameter
$$
t_\star = 8 \sigma_\star.
$$

The following claim is our tool to make a good choice for $t_\eps$.

\begin{claim} \label{claimchoiceofM}
Let $\alpha$ and $r_\eps$ be as above and 
$2\alpha < \tau < 1 $.  
For any $\eta>0$, there is $\eps_0 >0$ such that for any $0<\eps <\eps_{0}$ and any $\eps^{1-\tau}<  \xi_{\eps} < 1 - \eta $, there is $t_{\eps}>0$
(converging to $t_\star$ as $\eps \to 0$) such that $M = \xi_{\eps} $ for some first eigenfunction $u_\eps$.
\end{claim}

\begin{rem}
We have not ruled out the possibility that the first eigenvalue of $\Sigma_\eps$ has multiplicity.
In particular note that the assignment $t \mapsto M$ might not be a well-defined map but depends on the choice of a normalized first eigenfunction.
\end{rem}

\begin{proof}
We define a function $\tilde M \colon [t_0,t_1] \to [0,1]$ as follows.
 $$
 \tilde M = \inf_{\{u \in E_1 \ : \ \|u\|_{L^2(\partial \Sigma_\eps)}=1\}} \int_{I^+_\eps \cup I^-_\eps} |u|^2\, dl_\eps
 $$
 where $E_1$ denotes the space of $\sigma_\eps^1$-eigenfunctions (note that all of this depends on the parameter $t$ as well).
 
In a first step we argue that 
$$
 \tilde M(t_1) \leq \eps^{1-\tau}
$$
for $t_1>t_\star$ fixed and $\eps \leq \eps_0$.

If for $t_1>t_{\star}$, we have $\tilde M(t)> \eps^{1-\tau}$ along a sequence $\eps_n \to 0$, 
then
$$
\frac{\left(\ln r\right)^2 \eps \overline{\phi}(1)^2}{2M} \leq C \eps^{\tau - 2\alpha}  \left(\ln\eps\right)^2 \to 0
$$
for some constant $C$ as $\eps \to 0$ by \eqref{eqestneibpi} and \eqref{eqcompatibilityp0}. 
Using also the inequality \eqref{eqgapsigmastarsigmaeps2}, where $\alpha< \frac{1}{2}$, combined with \eqref{ineqsigmaeps}, we find from \eqref{eqenergyeigenfunc5}
that $\left(\ln r\right)^2\left(\frac{\sigma^1_{\eps}}{t_{\eps}} - \frac{1}{8} \right)$ is bounded from below. 
It is clear by \eqref{ineqsigmaeps} that it is also bounded from above by $\pi^2$. 
Therefore, it converges up to a subsequence to a constant $\lambda \in [0,\pi^2]$.

Therefore, we have that for any $t$ in a compact neighbourhood of $t_{\star}$,
$$
 \sigma_{\eps}^1 = t \left( \frac{1}{8} + \frac{\lambda}{2 \left(\ln r\right)^2} + o\left(\frac{1}{ \left(\ln r\right)^2}\right) \right) \hskip.1cm.
 $$
which is impossible by the eigenvalue bound from \eqref{ineqsigmaeps} for $\eps$ sufficiently small.

Next, we claim that 
 \begin{align} \label{eq_tilde_M_bd_1}
\lim_{\eps \to 0} \tilde M (t_0) = 1.
 \end{align}
For $t_0 < t_\star$ fixed we take $u_\eps$ a normalized first eigenfunction.
We have that $u_{\eps}$ is bounded in $W^{1,2}(\Sigma)$ since
$$
\int_\Sigma |\nabla u_\eps| \leq \sigma_\eps^1 \leq \sigma_\star + o(1)
$$
thanks to \eqref{ineqsigmaeps}, and
$$
\int_{\partial \Sigma} |u_\eps|^2 \leq N + C \eps^2 \left(\ln \eps\right)^2
$$
by \eqref{eqestneibpi}.
Therefore, it follows from standard Sobolev trace theory that  $(u_{\eps})$ is bounded $W^{1,2}(\Sigma)$.
Hence, after potentially taking a subsequence, 
$u_\eps$ converges weakly in $W^{1,2}(\Sigma)$ and strongly in $L^2(\partial{\Sigma})$ to a function $u_0$ on $\Sigma$. 
Moreover, by ellipticity of the eigenvalue equation, $u_{\eps}$ converges in $C^2_{loc}(\Sigma\setminus \{p_0,p_1\})$ to $u_0$. 
Therefore, $u_0$ satisfies a Steklov eigenvalue equation on $\Sigma$ with eigenvalue $\lim_{\eps\to 0} \sigma_{\eps}^1$. 
Thanks to the inequality \eqref{eqineqsigmaeps}, we know that 
$
\lim_{\eps\to 0} \sigma^1_{\eps}\leq \frac{t}{8}.
$
Moreover, using that $\frac{\overline{\theta}(0)^2}{\ln \frac{1}{r}} = O(\eps\left(\ln\eps\right)^2)$ as $\eps\to 0$, that $M+N = 1$,
and the definition of $\delta_{\eps}$ in formula \eqref{eqenergyeigenfunc5}, 
we get that $\lim_{\eps \to 0} \sigma_{\eps}^1 \geq \frac{t}{8}$. 
Therefore, we arrive at
$$
\lim_{\eps\to 0} \sigma_{\eps}^1 = \frac{t}{8} \in (0,\sigma_\star),
$$
where $\sigma_{\star}$ is the first non-zero Steklov eigenvalue on $\Sigma$ and we conclude that $u_0 =0$. 
By strong convergence in $L^2(\partial \Sigma)$ we then have that $\lim_{\eps\to 0} N = 0$ so that $\lim_{\eps\to 0} M = 1$. 
Since the argument applies to any sequence of eigenfunctions we conclude \eqref{eq_tilde_M_bd_1}.

We can now decrease $\eps_0>0$ if necessary such that we have
$$
\tilde M(t_1) \leq \eps^{1-\tau} < \xi < 1-\eta \leq \tilde M(t_0)
$$
for any $\eps \in (0,\eps_0]$.
In particular, 
$$
t_\eps = \sup\{ t \in [t_0,t_1] \ : \ \tilde M > \xi \} \in (t_0,t_1)
$$
is well-defined for $\eps \in (0,\eps_0]$. 
We claim that for $t_\eps$ we can find a first eigenfunction on $\Sigma_\eps$ as desired.

In fact, by construction, we find sequences $r_n \nearrow t_\eps$ and $s_n \searrow t_\eps$ and associated normalized first eigenfunctions $v_n$ and $w_n$ respectively with
$$
\int_{I_\eps^+\cup I_\eps^-} |v_n|^2 \, dl_\eps \leq \xi \leq \int_{I_\eps^+\cup I_\eps^-} |v_n|^2 \, dl_\eps.
$$
Thanks to \eqref{ineqsigmaeps} and the compact emedding $W^{1,2}(\Sigma_\eps) \to L^2(\partial \Sigma_\eps)$ (note that this 
holds also along a sequence of parameters $t$ since all the norms are uniformly equivalent for different choices of $t \in [t_0,t_1]$ 
and $\eps>0$ fixed) we may assume $v_n \to v$ and $w_n \to w$ 
weakly in $W^{1,2}(\Sigma_\eps)$ and strongly in $L^2(\partial \Sigma_\eps)$.
We now have two options: If $v= \pm w$ we are done immediately.
In any other case we can easily find a linear combination of $v$ and $w$ with the desired property.
\end{proof}

\subsection{The asymptotic expansion}
We now provide the precise asymptotic expansion of the first eigenfunction on the cuspidal domain.

\bigskip

Recall our choice of $2\alpha < \tau < 1 $ and $r_{\eps} = \exp\left(-\frac{1}{\eps^{\alpha}}\right)$.
Moreover, thanks to \cref{claimchoiceofM} we may assume that $t_{\eps}$ is such that $M \geq \eps^{1-\tau}$ for some first eigenfunction.
All the results in this section refer to this specific eigenfunction.

\bigskip

Now, we aim at studying the convergence properties of $\bar{\theta}$. We first focus on the function $\theta$ defined on $\widetilde{\Omega}$. 
We know that $\phi$ satisfies

\begin{equation} \label{eqsteklovphi} \begin{cases}
\epsilon^4 \Delta_{g_{\epsilon}} \phi := \phi_{xx} +\epsilon^{2}\phi_{yy} = 0 &  \text{in} \ \Omega_{\eps}
\\
\frac{\partial_{\nu^{\pm}_{\epsilon}} \phi}{t_{\eps}} := \frac{\pm \phi_x - \eps^2 y \phi_y}{\eps^2 \sqrt{1+\eps^2y^2}}  = \frac{\sigma_{\eps}}{t_{\eps}} \phi &  \text{on} \ I_{\eps}^{\pm} \\
\end{cases} 
\end{equation}
where $I_{\eps}^{\pm} = \{ (x,y) \in \mathbb{R}^2 ; r_{\eps}\leq y\leq 1 ,  y = \pm \frac{x^2}{2} \}$ and 
$\nu^{\pm}_{\eps} = t_{\eps}\frac{\left(\pm 1, - \eps^2 y \right)}{\eps^2 \sqrt{1+\eps^2y^2}}$ is the outward pointing normal along $I_{\eps}^{\pm}$ with respect to $g_{\eps}$. 
Therefore, $\theta$ satisfies the following equation
\begin{equation} \label{eqsteklovontheta} 
\begin{cases}
r^{2v}\theta_{xx} +\epsilon^{2} \left( \frac{3\theta}{4} + \frac{2\theta_v}{\ln  \frac{1}{r} } +  \frac{\theta_{vv}}{\left(\ln\frac{1}{r} \right)^2} \right) = 0 &  \text{in} \ \widetilde{\Omega}
\\
\pm \theta_x + \eps^2 \left( \frac{\theta}{2} + \frac{\theta_v}{\ln\frac{1}{r} }\right)  = \eps^2 \sqrt{1+\eps^2 r^{2v}} \frac{\sigma_{\eps}}{t_{\eps}} \theta &  \text{if} \ x = \pm \frac{r^{2v}}{2}
\end{cases} 
\end{equation}
Thanks to \eqref{eqsteklovontheta}, 
$\overline{\theta_v} = r^{-2v} \int_{-\frac{r^{2v}}{2}}^{\frac{r^{2v}}{2}} \theta_v(x,v)dx$ satisfies the equation
\begin{equation}\label{eqbarthetav}
- \overline{\theta_v}' = (\ln r)^2 \left( \left( \theta\left(\frac{r^{2v}}{2},v\right) + \theta\left(-\frac{r^{2v}}{2},v\right)\right)  \left( \frac{\sigma_{\eps}}{t_{\eps}}\sqrt{1+\eps^2 r^{2v}} -\frac{1}{2} \right) +\frac{3}{4} \bar \theta \right) \hskip.1cm.
\end{equation}
That is, if we denote $\mu(v) = \left( 2\bar{\theta}(v) - \theta\left(\frac{r^{2v}}{2},v\right) - \theta\left(-\frac{r^{2v}}{2},v\right) \right) $, we have the equation
\begin{equation}\label{eqonbarthetav2}
- \overline{\theta_v}' = (\ln r)^2 \left( \left( 2\frac{\sigma_{\eps}}{t_{\eps}}- \frac{1}{4} \right)\bar{\theta} + 2\frac{\sigma_{\eps}}{t_{\eps}} \left(\sqrt{1+\eps^2 r^{2v}} - 1\right)\bar{\theta}  + \left(\frac{1}{2} - \frac{\sigma_{\eps}}{t_{\eps}}\sqrt{1+\eps^2 r^{2v}}\right)\mu \right) \hskip.1cm.
\end{equation}
We also have that
\begin{equation}\label{eqonbartheta}
\bar\theta' = \left(\ln \frac{1}{r}\right)\mu + \overline{\theta_v}.
\end{equation}
Moreover, recall the compatibility conditions \eqref{eqcompatibilityp0} and \eqref{eqcompatibilityp1} for $v=0$ and $v=1$
given by
\begin{equation*} 
\frac{1}{\eps^2}\int_{-\frac{\eps^2}{2}}^{\frac{\eps^2}{2}} \left(u_{\eps}\circ \varphi_0^{-1}\right)(x,0)dx = \overline{\phi}(1) = \frac{\sqrt{t_{\eps}}\overline{\theta}(0)}{\sqrt{\eps}\sqrt{\ln\frac{1}{r}}} \hskip.1cm,
\end{equation*}
\begin{equation*} 
\frac{1}{r^2\eps^2}\int_{-\frac{r^2\eps^2}{2}}^{\frac{r^2\eps^2}{2}} \left( u_{\eps}\circ \varphi_1^{-1}\right)(x,0)dx  = \overline{\phi}(r) = \frac{\sqrt{t_{\eps}}r^{-\frac{1}{2}} \overline{\theta}(1)}{\sqrt{\eps}\sqrt{\ln\frac{1}{r}}} \hskip.1cm.
\end{equation*}

Using these equations on $\bar{\theta}$ and $\overline{\theta_v}$ and that thanks to \eqref{eqestthetalinex} $\left\| \mu \right\|_{L^2(0,1)}$ is controlled by $\left\|\theta_x \right\|_{L^2(\widetilde{\Omega})}$, which has to be very small as $\eps\to 0$, 
we can perform an asymptotic analysis of the eigenvalue $\sigma_{\eps}^1$, and of the functions $\bar{\theta}$ and $\overline{\theta_v}$ in $W^{1,2}(0,1)$ associated to the corresponding eigenfunction.

\begin{claim} \label{claimassymptotic} 
We have the following asymptotic expansion of $\sigma_{\eps}^1$
\begin{equation}\label{eqassympexpsigma}
\frac{\sigma_{\eps}^1}{t_{\eps}} 
= 
\frac{1}{8} + \frac{\pi^2}{2\left(\ln r\right)^2} - \frac{c_1}{c_0} \frac{\pi}{\left(\ln r\right)^2} 
+ 
O\left( \frac{1}{(\ln r)^2}  \left(\frac{c_1}{c_0}\right)^3+\left\|\frac{\theta_x}{c_0} \right\|_{L^2(\widetilde{\Omega})} 
+\eps^2 \right)
\end{equation}
as $\eps\to 0$, the $W^{1,2}(0,1)$ asymptotic expansion of $\bar{\theta}$
\begin{equation}\label{eqassympexpbartheta}
\frac{\bar{\theta}}{c_0} =  f + \frac{c_1}{c_0} f_1 + \left(\frac{c_1}{c_0}\right)^2 f_2 + 
O\left(  \left(\frac{c_1}{c_0}\right)^3 + (\ln r)^2  \left\|\frac{\theta_x}{c_0} \right\|_{L^2(\widetilde{\Omega})}  \right)
\end{equation}
and the $W^{1,2}(0,1)$ asymptotic expansion of $\overline{\theta_v}$
\begin{equation}\label{eqassympexpbarthetav}
\frac{\overline{\theta_v}}{c_0} =  f' + \frac{c_1}{c_0} \left(f_1\right)' + \left(\frac{c_1}{c_0}\right)^2 \left(f_2 \right)' 
O\left(  \left(\frac{c_1}{c_0}\right)^3 + (\ln r)^2  \left\|\frac{\theta_x}{c_0} \right\|_{L^2(\widetilde{\Omega})} \right)
\end{equation}
where $c_0 = \sqrt{M}$, $c_1 = \bar{\theta}(0)$, and for $v\in [0,1]$, we define $f(v) = \sin \pi v$ and
\begin{equation} \label{eqdeff1} 
f_1(v) = (1-v)\cos\pi v - \frac{1}{2\pi}\sin\pi v 
\end{equation}
\begin{equation} \label{eqdeff2} 
f_2(v) = \left( -\frac{v^2}{2}+v - \left(\frac{1}{2}+\frac{3}{8\pi^2}\right) \right) \sin\pi v \hskip.1cm. 
\end{equation}

\end{claim}

\begin{rem} \label{rem_error_scales}
We show below (see \eqref{eqestonnormmu}) that $\left\| \frac{\theta_x}{c_0} \right\|_{L^2} =O(\eps \ln(r)^{-1}) = O(\eps^{1+\alpha})$.
Therefore, in terms of $\alpha$ we can rewrite the error terms as follows.
In \eqref{eqassympexpsigma} the error term is controlled by
$$
\frac{\eps^{\frac{1}{2}(3+\alpha)} \log(1/\eps)^3}{c_0^3} + \eps^{1+\alpha} 
$$
and the error terms in \eqref{eqassympexpbartheta} and \eqref{eqassympexpbarthetav} are controlled by
$$
\frac{\eps^{\frac{3}{2}(1-\alpha)}\log(1/\eps)^3}{c_0^3}+\eps^{1-\alpha} \hskip.1cm.
$$
\end{rem}

\bigskip

\begin{proof}
Using that $ M \geq \eps^{1-\tau}$, $2\alpha < \tau < 1 $, and the pointwise estimates \eqref{eqestneibpi} in combination with the compatibility condition\eqref{eqcompatibilityp0}, 
we get that 
$$
\frac{\left(\ln r\right)^2 \eps \overline{\phi}(1)^2}{2M} \leq C \eps^{\tau - 2\alpha}  \left(\ln\eps\right)^2 \to 0 \text{ as } \eps\to 0 \hskip.1cm.
$$ 
From this combined with \eqref{ineqsigmaeps}, \eqref{eqenergyeigenfunc5} and \eqref{eqgapsigmastarsigmaeps2} we then deduce that $(\ln r)^2 \left(  2\frac{\sigma_{\eps}}{t_{\eps}}- \frac{1}{4} \right)$ is bounded from below. 
By \eqref{ineqsigmaeps}, it is also bounded from above by $\pi^2$. 
Therefore, up to taking a non-relabeled subsequence\footnote{We ignore the issue of taking subsequence from here on, since we only work with precompact sequences with limit independent of the subsequence.}, we may assume that 
$$
(\ln r)^2 \left(  2\frac{\sigma_{\eps}}{t_{\eps}}- \frac{1}{4} \right) \to \lambda \in [0,\pi^2].
$$ 
We then deduce from \eqref{eqestthetalinex} and \eqref{eqenergyeigenfunc5} that 
\begin{equation} \label{eqestonnormmu} 
\left\| \frac{\mu}{\sqrt{ M}} \right\|_{L^2(0,1)} 
\leq
\left\| \frac{\theta_x}{c_0} \right\|_{L^2(0,1)} 
=
O\left(\eps \left(\ln \frac{1}{r}\right)^{-1} \right) \text{ as } \eps \to 0 \hskip.1cm.
\end{equation}
Moreover, by \eqref{eq_bar_theta_l2} we also have that $\frac{\bar{\theta}}{\sqrt{ M}}$ is bounded in $L^2(0,1)$. 
By \eqref{eqenergyeigenfunc5} and Jensen's inequality we find that $\frac{\overline{\theta_v}}{\sqrt{ M}}$ is bounded in $L^2(0,1)$.
Thanks to \eqref{eqonbarthetav2}, we can then deduce that $\frac{\overline{\theta_v}}{\sqrt{ M}}$ is bounded in $W^{1,2}(0,1)$.
We may therefore take a subsequence such that
$$
\frac{\overline{\theta_v}}{\sqrt{ M}} \to g  
$$ 
weakly in $W^{1,2}(0,1)$ and strongly in $\mathcal{C}^{0,\frac{1}{2}}$.
 By \eqref{eqonbartheta} and \eqref{eqestonnormmu}, $\frac{\bar{\theta}}{\sqrt{M}}$ is bounded in $W^{1,2}(0,1)$ and we may thus assume that
 $$
 \frac{\bar{\theta}}{\sqrt{M}} \to f
 $$
 weakly in $W^{1,2}(0,1)$.
 Again, by \eqref{eqonbartheta} combined with \eqref{eqestonnormmu}, $\frac{\bar{\theta}}{\sqrt{M}}$ converges strongly in $W^{1,2}(0,1)$ to $f$ and $f' = g$. 
 Thanks to the pointwise bounds \eqref{eqestneibpi2} and \eqref{eqestneibpi} combined with the compatibility conditions \eqref{eqcompatibilityp0} and \eqref{eqcompatibilityp1}, and also using that that $M \geq \eps^{1-\tau}$, we have that $f(0) = 0$ and $f(1)= 0$. 
Therefore, passing to the limit in \eqref{eqonbarthetav2} we find that
\begin{equation} \label{eqsteklovfk} \begin{cases}
f'' + \lambda f =0 &  \text{in} \ (0,1) \\
f = 0 &  \text{if} \ v = 0 \ \text{or} \ v=1 \hskip.1cm.
 \end{cases} \end{equation}
Since $\int_0^1 f^2(v)dv = \frac{1}{2}$, we have that $f\neq 0$ and since $\lambda\leq \pi^2$, we must  have
$$
 \lambda = \pi^2 \hspace{0.2cm} \text{and} \hspace{0.2cm}
f(v) =\pm \sin(\pi v),
$$
i.e.\
\begin{align} \label{eqlimiteigenvalue_0}
\lim_{\eps \to 0} \left( (\ln r)^2 \left(  2\frac{\sigma_{\eps}}{t_{\eps}}- \frac{1}{4} \right) \right)= \pi^2
\end{align}
By changing $u_{\eps}$ to $-u_{\eps}$ if necessary we assume from here on that $f(v) = \sin(\pi v)$.

\bigskip 

Let us continue the asymptotic expansion of $\sigma_{\eps}$ and $\bar{\theta}$. We set 
$$ R_1 = \bar{\theta} -  c_0 f \hskip.1cm,$$ 
where $f(v) = \sin(\pi v)$ and $c_0 = \sqrt{M}$ and we set
$$
\nu_1 = \left(\ln r\right)^2\left( 2\frac{\sigma_{\eps}}{t_{\eps}} -\frac{1}{4}\right)  - \pi^2 + a \hskip.1cm,\
$$ 
where 
$$
 \left(\ln r\right)^2\left( 2\frac{\sigma_{\eps}}{t_{\eps}} -\frac{1}{4}\right)  - \pi^2  \to 0 \hskip.1cm,\
$$
as $\eps\to 0$,
and
$$
a(v) = \left( \ln r \right)^2 2\frac{\sigma_{\eps}}{t_{\eps}} \left(\sqrt{1+\eps^2 r^{2v}}-1\right) = O\left( \left(\ln r\right)^2 \eps^2 \right)  \hskip.cm.
$$ 
Then, we write \eqref{eqonbarthetav2} as
\begin{equation} \label{eqonRv}
- \left( \overline{\theta_v} - c_0 f'   \right)' = \nu_1 \bar{\theta} + \pi^2 R_1 + b \mu
\end{equation}
and \eqref{eqonbartheta} as
\begin{equation} \label{eqonRRv}
\overline{\theta_v} - c_0 f' = R_1' - \left(\ln\frac{1}{r}\right) \mu
\end{equation}
where
$$
b(v) = \left(\ln r\right)^2 \left(\frac{1}{2}-\frac{\sigma_{\eps}}{t_{\eps}} \sqrt{1+\eps^2 r^{2v}}\right) \sim \frac{3}{8} \left(\ln r\right)^2 $$
as $\eps\to 0$
since $\frac{\sigma_\eps}{t_\eps} \to \frac{1}{8}$ by \eqref{eqlimiteigenvalue_0}.
 We recall that  $\left\| \frac{\mu}{c_0} \right\|_{L^2(0,1)} = O\left(\eps \left(\ln\frac{1}{r}\right)^{-1} \right)$ (see \eqref{eqestonnormmu}). 
 We set 
\begin{equation} \label{eqdefscaledelta}
\delta_1 = \left\vert \nu_1 \right\vert + \left(\ln r\right)^2 \left\| \frac{\mu}{c_0} \right\|_{L^2(0,1)} + \left\| \frac{R_1}{c_0} \right\|_{W^{1,2}(0,1)} + r^{\frac{1}{4}} \hskip.1cm. 
\end{equation}
We divide the equations \eqref{eqonRv} and \eqref{eqonRRv} by $\delta_1 c_0$ so that the right-hand-terms in \eqref{eqonRv} and \eqref{eqonRRv} stay bounded in $L^2(0,1)$ (by the definition of $\delta_1$).
This means that 
$\frac{1}{\delta_1 c_0}\left( \overline{\theta_v} - c_0f' \right)$
is bounded in $W^{1,2}(0,1)$ and we may assume
 $$
 \frac{1}{\delta_1 c_0}\left( \overline{\theta_v} - c_0f' \right) \to \rho^0
 $$
 weakly in $W^{1,2}$ and strongly in $\mathcal{C}^{0,\frac{1}{2}}$.
Then, by the definition of $\delta_1$ \eqref{eqdefscaledelta}, $\frac{R_1}{\delta_1 c_0}$ is bounded in $W^{1,2}(0,1)$ and we may assume
$$
\frac{R_1}{\delta_1 c_0} \to R_1^0
$$
weakly in $W^{1,2}(0,1)$.
Then, passing to the limit in \eqref{eqonRRv}, $\frac{R_1}{\delta_1 c_0}$ converges strongly in $W^{1,2}(0,1)$ and $\left(R_1^0\right) ' = \rho_0$. 
And we get, passing to the limit in \eqref{eqonRv} and \eqref{eqdefscaledelta},
\begin{equation} \label{eqlimR0}
- \left(R_1^0\right)'' = \nu_1^0 f + \pi^2 R_1^0 + \mu_1^0  \hskip.1cm,
\end{equation}
\begin{equation} \label{eqlimdelta0}
1 = \left\vert\nu_1^0\right\vert + \alpha_1 + \left\| R^0_1 \right\|_{W^{1,2}(0,1)} 
\end{equation}
where 
\begin{equation} \label{eqdefalllimits}
\nu_1^0 = \lim_{\eps\to 0} \frac{\nu_1 + r^{\frac{1}{4}}}{\delta_1} \hskip.1cm,\hskip.1cm 
\alpha_1 =\lim_{\eps\to 0}\frac{\left(\ln r\right)^2}{\delta_1} \left\| \frac{\mu}{c_0} \right\|_{L^2(0,1)}
\text{ and } 
\frac{b\mu}{c_0 \delta_1} \to \mu_1^0 \text{ in } L^2(0,1) \text{ as } \eps \to 0 \hskip.1cm.
\end{equation}
We recall here that $f(v) = \sin(\pi v)$. 
We also set 
$$
\beta_1 =\lim_{\eps\to 0} \frac{\bar{\theta}(0)}{\delta_1 c_0}.
$$ 
Notice that $\left\| \mu_1^0 \right\|_{L^2(0,1)} \leq \alpha_1$ and that $\beta_1 = R_1^0(0) \leq \left\| R_1^0 \right\|_{W^{1,2}(0,1)} $. 
Note that $R_1^0(1)=0$ since $R_1(1) = O(\eps^{\frac{1-3\alpha}{2}}r^{\frac{1}{2}} )$ as $\eps\to 0$ thanks to the pointwise bound \eqref{eqestneibpi2}
and $\delta_1 \geq r^{\frac{1}{4}}$
Therefore, testing \eqref{eqlimR0} against $f$ we get 
\begin{equation} \label{eqnu0beta0mu0} 
-\pi\beta_1 = \frac{\nu_1^0}{2}  + \int_{0}^1 \mu_1 f dv = 0 \hskip.1cm. 
\end{equation}
Let us also prove that
\begin{equation} \label{eqorthogonalityR1f} 
\int_0^1 \frac{R_1}{\delta_1c_0}(v)f(v)dv \to 0 \text{ as } \eps \to 0 \hskip.1cm. 
\end{equation} 
We have
$$ 
\left\vert \int_0^1 \frac{R_1}{ c_0}(v)f(v)dv \right\vert  = \left\vert \int_0^1 \left(\frac{P(\bar{\theta})}{c_0} -  f \right)(v)f(v)dv  \right\vert 
\leq 
\left\| f \right\|_{L^{2}(0,1)}  \left\| \frac{P\left(\overline{\theta}\right)}{c_0} - f \right\|_{L^{2}(0,1)} \hskip.1cm,
$$
where we denote by $P(\bar{\theta})$ the orthogonal projection on the space generated by $f$ in $L^2(0,1)$.
A simple application of Pythagoras' theorem gives that
\begin{equation} \label{eqdiffpthetabarf}
\begin{split}\left\| \frac{P(\overline{\theta})}{c_0} - f \right\|_{L^2(0,1)}
\ = 
\frac{1}{2} \left\| \frac{\overline{\theta}}{c_0}-f \right\|_{L^2(0,1)}^2 + O\left( \left\| \frac{\theta_x}{c_0} \right\|_{L^2(0,1)} \right)
\\  = 
\frac{1}{2}\left\| \frac{R_1}{c_0} \right\|_{L^2(0,1)}^2 + O\left(\left\| \frac{\theta_x}{c_0} \right\|_{L^2(0,1)} \right) \hskip.1cm.
\end{split}
\end{equation}
In particular, by the definition of $\delta_1$, see \eqref{eqdefscaledelta}, and using that $\left\|\frac{R_1}{c_0}\right\|_{L^2(0,1)} \to 0$ we obtain \eqref{eqorthogonalityR1f}. 

If $\alpha_1 = \beta_1 =0$, then $\mu_1^0=0$ (recall that $\|\mu_1^0\|_{L^2(0,1)} \leq \alpha_1$) and then, by \eqref{eqnu0beta0mu0}, also $\nu_1^0=0$. 
Then, by  \eqref{eqlimR0}, $\left(R_1^0\right)''+ \pi^2 R_1^0 = 0$.
Since also $R_1^0(1)=0$ this implies that $R_1^0$ is parallel to $f$.
But by \eqref{eqorthogonalityR1f} $R_1^0$ is also orthogonal to $f$, which implies that we have $R_1^0=0$,
so that all terms on the right hand side of \eqref{eqlimdelta0} vanish, which is impossible. 

Therefore, we must have $\alpha_1 + \left\vert \beta_1 \right\vert \neq 0$. 
Notice that if $\alpha_1 \neq 0$, then, we have that
$$ 
\frac{\bar{\theta}(0)}{c_0} \sim \frac{\beta_1}{\alpha_1} \left(\ln r\right)^2 \left\| \frac{\mu}{c_0}\right\|_{L^2(0,1)} 
$$
so that the asymptotic expansion is proved. 

We now assume that $\alpha_1 = 0$. 
Therefore, by \eqref{eqdefscaledelta} and the definition of $\nu_1^0$ and $\beta_1 \neq 0$ in \eqref{eqdefalllimits},
\begin{equation} \label{eqestonnu} 
\nu_1 
\sim 
\frac{\nu_1^0}{\left\vert \beta_1 \right\vert} \frac{\left\vert \bar{\theta}(0)\right\vert}{\sqrt{M}} + r^{\frac{1}{4}}
\sim
- 2\pi \frac{\bar{\theta}(0)}{c_0}  +r^{\frac{1}{4}}
\text{ as } \eps\to 0 \hskip.1cm. \end{equation}
Again, by \eqref{eqdefscaledelta} and \eqref{eqdefalllimits}, and because $R_0$ is the strong limit of $ \frac{R}{\delta_1 c_0}$ in $W^{1,2}(0,1)$,  
\begin{equation} \label{eqnormonrest}  \left\| \frac{R_1}{c_0} \right\|_{W^{1,2}(0,1)} \sim \frac{\left\| R_1^0 \right\|_{W^{1,2}(0,1)}}{ \left\vert \beta_1 \right\vert }\frac{\left\vert \bar{\theta}(0)\right\vert}{c_0}  \text{ as } \eps\to 0 \hskip.1cm. \end{equation}
Notice also that $ \frac{R_1 - \bar{\theta}(0)f_1}{\delta_1 c_0}$ converges to $0$ in $W^{1,2}(0,1)$ where $f_1$ defined by \eqref{eqdeff1} is the unique solution of the equation $ -f_1'' - \pi^2 f_1 = -2\pi f$ which satisfies $f_1(0) = 1$, $f_1(1) = 0$ and $\int_{0}^1 f_1(v)f(v)dv = 0$.

\bigskip

Therefore, we can continue the asymptotic expansion, we set 
$$ 
R_2 =  \bar{\theta} - c_0 f - c_1 f_1 \hskip.1cm,
$$
where we recall that $c_0 = \sqrt{M}$ and we set $c_1 = \bar{\theta}(0)$. 
Knowing \eqref{eqestonnu}, we set
$$
\nu_2 = \nu_1 + 2\pi \frac{c_1}{c_0} \hskip.1cm.
$$
Then, we write \eqref{eqonRv} as
\begin{equation} \label{eqonR1v}
- \left(\overline{\theta_v} - c_0 f' - c_1 f_1' \right)' = \nu_2 \bar{\theta} + \pi^2 R_2 - 2\pi \frac{c_1}{c_0} R_1 + b \mu 
\end{equation}
and \eqref{eqonRRv} as
\begin{equation} \label{eqonR1R1v}
\overline{\theta_v} - c_0 f' - c_1 f_1' = R_2' + \left(\ln\frac{1}{r}\right) \mu \hskip.1cm.
\end{equation}
We set 
\begin{equation} \label{eqdefscaledelta1}
\delta_2 = \left\vert \nu_2 \right\vert + \left(\ln r\right)^2 \left\| \frac{\mu}{c_0} \right\|_{L^2(0,1)} + \left\| \frac{R_2}{c_0} \right\|_{W^{1,2}(0,1)} + \frac{\left(c_1\right)^2}{\left( c_0\right)^2} +r^{\frac{1}{4}}  \hskip.1cm. 
\end{equation}
Then, by the the same argument as above, we can pass to the limit in \eqref{eqonR1v} and \eqref{eqonR1R1v} and find that
\begin{equation} \label{eqlimR20}
- \left(R_2^0\right)'' = \nu_2^0 f + \pi^2 R_2^0 - 2\pi \beta_2 f_1 + \mu_2  \hskip.1cm,
\end{equation}
\begin{equation} \label{eqlimdelta20}
1 = \left\vert\nu_2^0 \right\vert + \alpha_2 + \left\| R_2^0 \right\|_{W^{1,2}(0,1)} + \beta_2 \hskip.1cm,
\end{equation}
where $R_2^0$ is the strong limit of $\frac{R_2}{\delta_2 c_0}$ in $W^{1,2}$,
\begin{equation} \label{eqdefalllimits2}
\nu_2^0 = \lim_{\eps\to 0} \frac{\nu_2 + r^{1/4} }{\delta_2} \hskip.1cm,\hskip.1cm \alpha_2 =\lim_{\eps\to 0}\frac{\left(\ln r\right)^2}{\delta_2} \left\| \frac{\mu}{c_0} \right\|_{L^2(0,1)} \hskip.1cm,\hskip.1cm \beta_2 = \lim_{\eps\to 0} \frac{\left(c_1\right)^2}{\delta_2 \left(c_0\right)^2} 
\end{equation}
and 
$$ 
\frac{b\mu}{\delta_2 c_0} \to \mu_2 \text{ in } L^2(0,1) \text{ as } \eps \to 0 \hskip.1cm.
$$
Now, if $\alpha_2 \neq 0$, we get that
$$ 
\frac{\left(c_1\right)^2}{\left(c_0\right)^2} \sim \frac{\beta_2}{\alpha_2} \left(\ln r\right)^2 \left\| \frac{\mu}{c_0}\right\|_{L^2(0,1)} 
$$
and the asymptotic expansion is proved.
 
We now assume that $\alpha_2 = 0$. 
We integrate \eqref{eqlimR20} against to $f$ and because $R_2^0(0)= R_2^0(1)=0$, and $\mu_2 =0$, we get
\begin{equation} \label{eqnu2beta2mu2}
0 = \frac{\nu_2^0}{2}     \hskip.1cm. 
\end{equation}
Therefore, \eqref{eqlimR20} becomes $ - \left(R_2^0\right)'' - \pi^2 R_2^0 = -2\pi \beta_2 f_1 $ and we have more precisely that $\frac{R_2 - c_2 f_2}{\delta_2 c_0}$ converges to $0$ in $W^{1,2}$, 
where $c_2 = \frac{c_1^2}{c_0}$, and $f_2$ is defined by \eqref{eqdeff2} and is the solution of 
$$
 - f_2'' - \pi^2 f_2 = - 2\pi f_1 
 $$
such that $f_2(0) = 0$, $f_2(1) = 0$ and $\int_{0}^1 f_2 f + \frac{1}{2} \int_0^1 \left(f_1\right)^2= 0 $. 
Indeed, because $\int_0^1 f_1 f = 0$, we have that 
\begin{equation} \label{eqintR2R1} 
\int_0^1 R_2 f  = \int_0^1 R_1f = c_0 \int_0^1 \left(\frac{P(\bar{\theta})}{c_0}-f\right) f = -\frac{c_0}{\sqrt{2}} \left\| \frac{P\left(\overline{\theta}\right)}{c_0} - f \right\|_{L^{2}(0,1)}
 \end{equation}
and thanks to \eqref{eqdiffpthetabarf} and \eqref{eqnormonrest}, we get
$$ \int_0^1 R_2 f = -\frac{c_0}{2} \left\| \frac{R_1}{c_0} \right\|_{L^{2}(0,1)}^2
$$ 
so that we have the integral formula on $f_2$ when we pass to the limit. 

We can continue the asymptotic expansion again. 
We set 
$$
R_3 =  \bar{\theta} - c_0 f - c_1 f_1 - c_2 f_2 \hskip.1cm,
$$
and we set 
$$
\nu_3 = \nu_2 = \nu_1 + 2\pi \frac{c_1}{c_0} \hskip.1cm.
$$
Then, we write \eqref{eqonR1v} as
\begin{equation} \label{eqonR3v}
- \left(\overline{\theta_v} - c_0 f' - c_1f_1' - c_2 f_2' \right)' = \nu_3 \bar{\theta} + \pi^2 R_3 - 2\pi \frac{c_1}{c_0}R_2   + b \mu 
\end{equation}
and \eqref{eqonR1R1v} as
\begin{equation} \label{eqonR3R3v}
\overline{\theta_v} - c_0 f' - c_1f_1' - c_2 f_2' = R_3' + \left(\ln\frac{1}{r}\right) \mu 
\end{equation}
We set 
\begin{equation} \label{eqdefscaledelta3}
\delta_3 = \left\vert \nu_3 \right\vert + \left(\ln r\right)^2 \left\| \frac{\mu}{c_0} \right\|_{L^2(0,1)} + \left\| \frac{R_3}{c_0} \right\|_{W^{1,2}(0,1)} + \frac{\left(c_1\right)^3}{ \left(c_0\right)^3} +r^{\frac{1}{4}}  \hskip.1cm. 
\end{equation}
Then, by the same arguments as above, one can pass to the limit in \eqref{eqonR3v} and \eqref{eqonR3R3v} in order to have
\begin{equation} \label{eqlimR30}
- \left(R_3^0\right)'' = \nu_3^0 f + \pi^2 R_3^0 - 2\pi \beta_3 f_1 + \mu_3  \hskip.1cm,
\end{equation}
\begin{equation} \label{eqlimdelta30}
1 = \left\vert\nu_3^0 \right\vert + \alpha_3 + \left\| R_3^0 \right\|_{W^{1,2}(0,1)} + \beta_3 \hskip.1cm,
\end{equation}
where $R_3^0$ is the strong limit of $\frac{R_3}{\delta_3 c_0}$ in $W^{1,2}$,
\begin{equation} \label{eqdefalllimits3}
\nu_3^0 = \lim_{\eps\to 0} \frac{\nu_3 + r^{\frac{1}{4}}}{\delta_3} \hskip.1cm,\hskip.1cm \alpha_3 =\lim_{\eps\to 0}\frac{\left(\ln r\right)^2}{\delta_3} \left\| \frac{\mu}{c_0} \right\|_{L^2(0,1)} \hskip.1cm,\hskip.1cm \beta_3 = \lim_{\eps\to 0} \frac{\left(c_1\right)^3}{\delta_2 \left(c_0\right)^3} 
\end{equation}
and 
$$ \frac{b\mu}{\delta_3 c_0} \to \mu_3 \text{ in } L^2(0,1) \text{ as } \eps \to 0 \hskip.1cm.$$
Now, if $\alpha_3 \neq 0$, we get that
$$ \frac{\left(c_1\right)^3}{\left(c_0\right)^3} \sim \frac{\beta_2}{\alpha_2} \left(\ln r\right)^2 \left\| \frac{\mu}{c_0}\right\|_{L^2(0,1)} $$
and the asymptotic expansion is proved. 

We now assume that $\alpha_3 = 0$. 
We integrate \eqref{eqlimR20} against $f$ and because $R_3^0(0)= R_3^0(1)=0$, and $\mu_3 =0$, we get
\begin{equation} \label{eqnu3beta3mu3} 
0 = \frac{\nu_2^0}{2} - 2\pi \int_0^1 f f_2     \hskip.1cm. 
\end{equation}
Therefore, \eqref{eqlimR20} becomes $ - \left(R_2^0\right)'' - \pi^2 R_2^0 = -2\pi \beta_2 f_1 $ and we have more precisely that $\frac{R_2 - c_3 f_3}{\delta_3 c_0}$ converges to $0$ in $W^{1,2}$, 
where $c_3 = \frac{\left(c_1\right)^3}{\left(c_0\right)^2}$, and $f_3$ is a solution of 
$$ 
- f_3'' - \pi^2 f_3 =  - 2\pi \left( f_2 - 2 \left(\int_{0}^1 f f_2\right) f \right) \hskip.1cm,
$$
which gives the asymptotic expansion.
\end{proof}

\bigskip

We know aim at applying \eqref{eqgapsigmastarsigmaeps2} with good estimates of the right-hand side. 
Going back to the notation $\phi$ of the eigenfunction associated to $\sigma_{\eps}$ on the chart $\Omega_{\eps}$, 
we first have to estimate its mean value on $I^+\cup I^-$. By taking the derivative $\phi_y$ with respect to $y$ of $\phi$ 
and then the mean value with respect to $x$, on the equation  \eqref{eqdeftheta} and denoting $y=r^v$,  we have
\begin{equation}
\label{eqderivativephibar}
\overline{\phi_y}(y) = - \frac{\sqrt{t_{\eps}}y^{-\frac{3}{2}}}{\sqrt{\eps}\sqrt{\ln\frac{1}{r}}}\left( \frac{\bar{\theta}(v)}{2} + \frac{\overline{\theta_v}(v)}{\ln\frac{1}{r}} \right) \hskip.1cm.
\end{equation}
Since $\phi$ is harmonic on $\Omega_{\eps}$, we have by integration by parts
\begin{equation} \label{eqintlaplacephi}
\eps \left( \overline{\phi_y}(1) - r^2  \overline{\phi_y}(r) \right) = - \int_{I^+\cup I^-} \partial_{\nu_{\eps}^{\pm}}\phi dl_{\eps} = - \sigma_{\eps} \int_{I^+\cup I^-} \phi dl_{\eps} \hskip.1cm.
\end{equation}
By \cref{claimassymptotic} we know that $\frac{\overline{\theta_v}}{c_0}$ converges to $f'(v) = \pi\cos(\pi v)$ in $C^{0,\frac{1}{2}}$
so that by \eqref{eqderivativephibar}, we have that
$$ 
\overline{\phi_y}(1) \sim -\left( \frac{1}{2}\overline{\phi}(1) + \sqrt{t_{\eps}}\pi c_0 \eps^{\frac{3\alpha}{2}-\frac{1}{2}} \right) \hskip.1cm,
$$
$$
\overline{\phi_y}(r) \sim - r^{-\frac{3}{2}} \left(\frac{r^{\frac{1}{2}}}{2}\overline{\phi}(r) - \sqrt{t_{\eps}}\pi c_0 \eps^{\frac{3\alpha}{2}-\frac{1}{2}} \right) 
$$
as $\eps \to 0$. 
Therefore,
$$
 \bar{\phi}(1)  \sim  2\left(\frac{\sigma_{\eps}}{\eps}\int_{I^+\cup I^-} \phi dl_{\eps} +  r\frac{\bar{\phi}(r)}{2} - (1+r^{\frac{1}{2}}) \sqrt{t_{\eps}}\pi c_0 \eps^{\frac{3\alpha}{2}-\frac{1}{2}} \right) 
 $$
so that 
\begin{equation} \label{eqmeanvaluephi}  
\int_{I^+\cup I^-} \phi dl_{\eps} \sim \eps \frac{\bar{\phi}(1)}{2\sigma_{\eps}} + \frac{\sqrt{t_{\eps}}\pi c_0}{\sigma_{\eps}} \eps^{\frac{3\alpha}{2}+\frac{1}{2}} 
\end{equation}
as $\eps\to 0$.

\bigskip

Now, as already said, we aim at obtaining a good estimate for $\delta_{\eps}$ in order to apply \eqref{eqgapsigmastarsigmaeps2}. 
We have by \eqref{eqcompatibilityp0} and \eqref{eqenergyeigenfunc5} that
\begin{equation} \label{eqenergydeltaeps}
\begin{split} \frac{\left\| \theta_x \right\|_{L^2(0,1)}^2}{\eps^2}+ \int_0^1 \frac{\overline{\left( \theta_v - \overline{\theta_v}\right)^2} }{\left(\ln r\right)^2} + \frac{\delta_{\eps}}{t_{\eps}} 
= 
c_0^2\left( \frac{\sigma_{\eps}}{t_{\eps}} - \frac{1}{8} - \frac{\pi^2}{2\left(\ln r\right)^2} + \frac{c_1}{c_0}\frac{\pi}{ \left(\ln r\right)^2} \right) \\ 
+ \frac{\left(c_0\right)^2}{\left(\ln r\right)^2} \left( \frac{\pi^2}{2} - \frac{c_1}{c_0}\pi  - \int_0^1 \overline{\left( \frac{\theta_v}{c_0} \right)}^2 dv \right) + \frac{\bar{\theta}(0)^2}{2\ln\frac{1}{r}} + O(\eps^{1+\alpha}) \hskip.1cm,
\end{split}
\end{equation}
where we chose $\eta_{\eps} = \eps^{1+\alpha}$ in \eqref{eqestimateerrore}, in order to have $e_{\eps}+r\eps\left(\ln \frac{1}{r\eps}\right)^2 = O(\eps^{1+\alpha})$ as $\eps\to 0$. 
By \eqref{eqassympexpsigma}, we have that
$$ 
c_0^2\left( \frac{\sigma_{\eps}}{t_{\eps}} - \frac{1}{8} - \frac{\pi^2}{2\left(\ln r\right)^2} + \frac{c_1}{c_0}\frac{\pi}{ \left(\ln r\right)^2} \right) 
=  
O\left(\eps^{2\alpha} \frac{\left(c_1\right)^3}{c_0} + c_0 \left\| \theta_x \right\|_{L^2} +c_0^2 \eps^2 \right) \text{ as } \eps\to 0\hskip.1cm.
$$
Finally, by \eqref{eqassympexpbarthetav},
\begin{eqnarray*}
\left\| \frac{\overline{\theta_v}}{c_0} - f' - \frac{c_1}{c_0}f_1' - \frac{\left(c_1\right)^2}{\left(c_0\right)^2} f_2' \right\|_{W^{1,2}(0,1)} & = & O\left( \frac{\left(c_1\right)^3}{\left(c_0\right)^3} + \eps^{-2\alpha} \left\| \frac{\theta_x}{c_0} \right\|_{L^2} \right)
\hskip.1cm.
\end{eqnarray*}
and we have  $\int_{0}^1 \left(f'\right)^2 = -\pi $ and $\int_{0}^1 f'\left(f_1\right)' = -\pi  $, so that
$$
 \frac{c_0^2}{\left(\ln r\right)^2} \left( \frac{\pi^2}{2}   - \frac{c_1}{c_0}\pi - \int_0^1 \overline{\left( \frac{\theta_v}{c_0} \right)^2 } dv \right) \leq  O\left( \eps^{2\alpha}  \left(c_1\right)^2 + c_0 \left\| \theta_x \right\|_{L^2}  \right) \hskip.1cm.
 $$
Gathering all the previous inequalities and $\eqref{eqmeanvaluephi}$, in order to estimate $\delta_{\eps}$ in \eqref{eqenergydeltaeps}, 
\begin{equation} \label{eqenergydeltaeps2}
\begin{split} 
\frac{\left\| \theta_x \right\|_{L^2(\widetilde{\Omega})}^2}{\eps^2}+ \int_0^1 \frac{\overline{\left( \theta_v - \overline{\theta_v}\right)^2} }{\left(\ln r\right)^2} + \frac{\delta_{\eps}}{t_{\eps}} = \frac{\bar{\theta}(0)^2}{2\ln\frac{1}{r}} + 
O\left( \eps^{2\alpha}  \left(c_1\right)^2 + c_0 \left\| \theta_x \right\|_{L^2} +\eps^{1+\alpha} \right) 
\end{split}
\end{equation}
and then, the inequality \eqref{eqgapsigmastarsigmaeps2} becomes
\begin{equation} \label{eqineqsigmastarordereps} 
\sigma_{\star} - \sigma_{\eps} \leq \eps\frac {\bar{\phi}(1)^2}{2 N} +  O\left( \frac{1}{N} \left( \eps^{2\alpha} \left(c_1\right)^2 + c_0 \left\| \theta_x \right\|_{L^2} +\eps^{1+\alpha} \right)\right)  \hskip.1cm, 
\end{equation}
where we have that $t_{\eps}\eps^{2\alpha}c_1^2 = \epsilon^{1+\alpha} \bar{\phi}(1)^2 $. 
We now choose for the rest of the argument 
$$
M = N = \frac{1}{2},
$$
which is possible thanks to Claim \ref{claimchoiceofM}. 
If we can prove that $\bar{\phi}(1) = o(1)$ as $\eps\to 0$, we can conclude the proof of \cref{thm_glue}. 
Indeed, \eqref{eqineqsigmastarordereps} would give that $\sigma_{\star} - \sigma_{\eps} = o(\eps)$ and the extra-length of $\Sigma_{\eps}$ with respect to $\Sigma$ is of order $\eps$. 

However, there is in general no reason that $\bar{\phi}(1) = o(1)$ as $\eps\to 0$. 
Indeed, using \eqref{eqenergydeltaeps2} we can prove that the $W^{1,2}$ and energy norms in the right-hand term of \eqref{eqestneibp0precise} converge to $0$ as fast as $\eps^{\frac{1}{2}}$ (see formula \eqref{eqestonc_1c_0} in Claim \ref{claimc_1c_0d_1d_0}, in the next section). 
Therefore, $\bar{\phi}(1)$ converges to $u_{\star}(p_0)$, where $u_{\star}$ is the weak limit of $u_{\eps}$ in $W^{1,2}\left(\Sigma\right)$. As a $\sigma_{\star}$-eigenfunction on $\Sigma$, $u_{\star}$ does not necessarily vanish at $p_0$. 
In the case when $\sigma_{\star}$ is simple, one can choose an attaching point $p_0$ such that $u_{\star}(p_0)=0$. 
More generally, if any eigenfunction associated to $\sigma_{\star}$ vanishes at some point $p_0$, we get the theorem.
This is not necessarily true. 

Therefore, the Poincar\'e inequality we invoked to prove \eqref{eqgapsigmastarsigmaeps} and then \eqref{eqgapsigmastarsigmaeps2} is not sufficient to get $\sigma_{\star} - \sigma_{\eps} = o(\eps)$. 
The key idea to improve this estimate is to use a better perturbation of a first eigenfunction to get good control on the mean value.
(Before, we simply did this by a constant function.)
The way we do this is by using another eigenfunction with eigenvalue close to $\sigma_\eps^1$ which perturbs the eigenvalue equation on a much smaller scale.

\section{The improved test function} \label{section6}

Recall all the choices up to this point.
We have $r_{\eps} = \exp\left(-\frac{1}{\eps^{\alpha}}\right)$, with $\alpha<\frac{1}{2}$ and $t_\eps$ chosen with the help of
\cref{claimchoiceofM} such that we have
 $u_\eps^1$ a normalized $\sigma_\eps^1$-eigenfunction with $M=N=\frac{1}{2}$.
 From here on we assume in addition that also $\alpha>\frac{1}{3}$ such that $\eps^{\frac{3\alpha}{2}+\frac{1}{2}}=o(\eps)$
 for the scale in the estimates above.

In order to conclude our main result we would like to choose a better test function by finding a good linear combination of 
the first eigenfunction and another eigenfunction with eigenvalue close to $\sigma_\eps^1$.
We would like to have a similar asymptotic expansion on the cuspidal domain available for the corresponding eigenfunction in order to
arrange for cancelations in our asymptotic estimates.
In order to do so we first need to locate another eigenfunction with some concentration of mass on the cuspidal domain.

Let us consider an eigenvalue $\sigma_{\eps}^{l}$ on $\Sigma_{\eps}$ with $2 \leq l \leq K+1$ where we recall that $K=\mult \sigma_\star$ denotes the multiplicity of
the first non-trivial eigenvalue on $\Sigma$.
In practice, $l$ could depend on $\eps$ but up to taking a subsequence we may assume that it is fixed.

By the estimates from the  previous sections, we can prove that
\begin{equation} \label{eqestonsigma2ordereps} 
\sigma_{\eps}^l =  \frac{t_{\eps}}{8} + \frac{t_{\eps} \pi^2}{2} \eps^{2\alpha} + O\left(\eps \left(  \ln\eps\right)^2 \right) 
\end{equation}
as $\eps\to 0$. 
Indeed, by \eqref{ineqsigmaeps2}, if $\sigma_{\eps}^l$ is controlled by the second term in the maximum, \eqref{eqestonsigma2ordereps} holds true (as we know the lower bound even for $\sigma_\eps^1$).
 If not, we have that $\sigma_{\eps}^l - \sigma_{\star} = O\left(\eps^{\frac{3\alpha}{2}+\frac{1}{2}}\right)$ as $\eps\to 0$. 
 Therefore, using \eqref{eqestneibpi} and \eqref{eqineqsigmastarordereps}, we have
$$
 \sigma_{\eps}^l = \sigma_{\eps}^l - \sigma_{\star} + \sigma_{\star} - \sigma_{\eps}^1 + \sigma_{\eps}^1 = \sigma_{\eps}^1 +  O\left(\eps  \left(\ln \eps\right)^2 \right)  \text{ as } \eps\to 0 \hskip.1cm,
 $$
and using \eqref{eqassympexpsigma}, we get the conclusion \eqref{eqestonsigma2ordereps}. 
In order to obtain the same type of asymptotic expansion for $\sigma_\eps^l$ and the corresponding eigenfunction as in \cref{claimassymptotic} we need to make sure that the eigenfunction has some of its mass
concentrated on the cuspidal domain.
We achieve this by choosing $l$ appropriately in the next claim.

\begin{claim}  \label{claim_choice_l}
For $\eps>0$ sufficiently small we can find $l \in \{2,\dots, K+1\}$ such that there is a normalized $\sigma_\eps^l$-eigenfunction orthogonal to $u_\eps^1$ with
$$
\int_{I_\eps^+ \cup I_\eps^-} |u_\eps^l|^2 dl_\eps \geq \frac{1}{4K}.
$$
\end{claim}

\begin{proof}
We argue by contradiction and may thus take a collection $u_\eps^2,\dots u_\eps^{K+1}$ of orthonormal eigenfunctions all orthogonal to $u_\eps^1$, such that
$u_\eps^l$ is an $\sigma_\eps^l$-eigenfunction and
$$
\int_{I_\eps^+ \cup I_\eps^-} |u_\eps^l|^2 dl_\eps \leq \frac{1}{4K}
$$
for any $l=2,\dots,K+1$.
In particular, also using the corresponding bound for the first eigenfunction, we find by Cauchy--Schwarz for $w_\eps = \sum_{i=1}^{K+1}t_i u_\eps^i$ that
\begin{align} \label{eq_orthog}
\int_{I_\eps^+ \cup I_\eps^-} |w_\eps|^2 dl_\eps \leq
\int_{I_\eps^+ \cup I_\eps^-} \left( \sum_{i=1}^{K+1} t_i^2\right) \left( \sum_{i=1}^{K+1} |u_\eps^i|^2 \right)dl_\eps 
\leq \frac{3}{4} \|w_\eps\|_{L^2(\partial \Sigma_\eps)}^2 \hskip.1cm.
\end{align}
After taking appropriate subsequences we can assume that 
$u_\eps^l \to u_\star^l$ weakly in $W^{1,2}(\Sigma)$ and strongly in $L^2(\partial \Sigma)$ for $l=1,\dots,K+1$.
As in the proof of \cref{claimchoiceofM} we find that all of these are $\sigma_\star$-eigenfunctions $u_\star^l$.
But this can be seen to be impossible as follows.
We can choose $w_\eps = \sum_{i=1}^{K+1}t_i u_\eps^i$ with $\|w_\eps\|_{L^2(\partial \Sigma_\eps)}=1$ such that $\int_{\partial \Sigma} w_\eps u_\star^l=0$ for any $l=1,\dots,K+1$ since the multiplicity of $\sigma_\star$ is only $K$,
so these are in fact only $K$ linear conditions.
By strong convergence in $L^2(\partial \Sigma)$ we find that 
$
\|w_\eps\|_{L^2(\partial \Sigma)} \to 0
$
as $\eps \to 0$.
When combined with \eqref{eq_orthog} this gives
$$
\|w_\eps\|_{L^2(\partial \Sigma_\eps)}^2
\leq 
\int_{I_\eps^+ \cup I_\eps^-} |w_\eps|^2 dl_\eps + o(1)
\leq
\frac{3}{4} \|w_\eps\|_{L^2(\partial \Sigma_\eps)}^2 + o(1) \hskip.1cm,
$$
which is a contradiction for $\eps$ sufficiently small. 
\end{proof}

Thanks to \eqref{eqestonsigma2ordereps} and \cref{claim_choice_l} we have a good asymptotic expansion for $u_\eps^l$ on the cuspidal domain.
This is collected in the following claim,
where
we denote by $\theta_l$ the good representation of $u_{\eps}^l$ on the cuspidal domain given by definition \eqref{eqdeftheta}.

\begin{claim}
Let $l \in \{2,\dots,K+1\}$ be given by \cref{claim_choice_l}.
We have the following asymptotic expansion of $\sigma_{\eps}^l$
\begin{equation}\label{eqassympexpsigma2}
\frac{\sigma_{\eps}^l}{t_{\eps}} = \frac{1}{8} + \frac{\pi^2}{2\left(\ln r\right)^2} - \frac{d_1}{d_0} \frac{\pi}{\left(\ln r\right)^2} +
O\left(\frac{1}{(\ln r)^2}\left(\frac{d_1}{d_0}\right)^3+\left\|\frac{(\theta_l)_x}{d_0} \right\|_{L^2(\widetilde{\Omega})}
+\eps^2 \right)
\end{equation}
as $\eps\to 0$, the $W^{1,2}(0,1)$ asymptotic expansion of $\overline{\theta_l}$
\begin{equation}\label{eqassympexpbartheta2}
\frac{\overline{\theta_l}}{d_0} =  f + \frac{d_1}{d_0} f_1 + \left(\frac{d_1}{d_0}\right)^2 f_2 + 
O\left( \left(\frac{d_1}{d_0}\right)^3 + (\ln r)^2 \left\|\frac{(\theta_l)_x}{d_0} \right\|_{L^2(\widetilde{\Omega})} \right)  
\end{equation}
and the $W^{1,2}(0,1)$ asymptotic expansion of $\overline{\left(\theta_{l}\right)_{v}}$
\begin{equation}\label{eqassympexpbarthetav2}
\frac{\overline{\left(\theta_{l}\right)_{v}}}{d_0} =  f' + \frac{d_1}{d_0} \left(f_1\right)' + \left(\frac{d_1}{d_0}\right)^2 \left(f_2 \right)' + 
O\left( \left(\frac{d_1}{d_0}\right)^3 + (\ln r)^2 \left\|\frac{(\theta_l)_x}{d_0} \right\|_{L^2(\widetilde{\Omega})} \right)  \hskip.1cm,
\end{equation}
where $d_0 = \sqrt{M_l} = \sqrt{\int_{I^+\cup I^-}\left(\phi_l\right)^2 dl_{\eps}}$ , $d_1 = \overline{\theta_l}(0)$, and $f$, $f_1$ and $f_2$ where defined previously in \eqref{eqdeff1} and \eqref{eqdeff2}.
\end{claim}

\bigskip

\begin{proof}
The proof is the same as the proof of Claim \ref{claimassymptotic}, 
noticing that $\overline{\theta_l}$ and $\overline{\left(\theta_l\right)_v}$ satisfy the same equation as $\overline{\theta_1}$ and $\overline{\left(\theta_1\right)_v}$ (see \eqref{eqonbarthetav2} and \eqref{eqonbartheta}, 
adding the indices $l$ on $\theta$ and $\sigma_{\eps}$), 
but it starts with the estimate \eqref{eqestonsigma2ordereps} which already gives that 
$$
\left(\ln r\right)^2 \left(\frac{2\sigma_{\eps}^l}{t_{\eps}} - \frac{1}{4}  \right) = \pi^2 + O\left( \eps^{1-2\alpha}\left( \ln \eps\right)^2 \right) \hskip.1cm,
$$
as $\eps\to 0$. 
We also have the identities \eqref{eqestthetalinex} and \eqref{eqenergyeigenfunc5} (adding the indices $l$ on $\theta$, $\delta_{\eps}$, $\sigma_{\eps}$) which give thanks to our choice of $l$ by \cref{claim_choice_l} that
\begin{equation} \label{eqestonnormmu2} 
\left\| \frac{\mu_l}{\sqrt{ M_l}} \right\|_{L^2(0,1)} = O\left(\eps \left(\ln \frac{1}{r}\right)^{-1} \right) \text{ as } \eps \to 0 \hskip.1cm,  \end{equation} 
Therefore, dividing the equations by $\sqrt{M_l}$, we have a $W^{1,2}$ convergence of $\frac{\overline{\theta_l}}{\sqrt{M_l}}$ to a non-zero solution of $f'' +\pi^2 f =0$. Our choice of $l$ by \cref{claim_choice_l} gives that $f(0)=f(1) = 0$.
 Up to changing $u_{\eps}^l$ into $-u_{\eps}^l$, we can assume that $f(v) = +\sin(\pi v)$.

Then, the proof of the asymptotic expansion of $\sigma_{\eps}^l$, $\overline{\theta_l}$ and $\overline{\left(\theta_l\right)_v}$ is exactly the same as the proof of the asymptotic expansion of $\sigma_{1}^{\eps}$, $\overline{\theta_1}$ and $\overline{\left(\theta_1\right)_v}$ in  \cref{claimassymptotic}. 
\end{proof}

\bigskip

We now need to give precise estimates on $c_0$, $d_0$, $c_1$ and $d_1$ in order to work with these parameters. 
We define by $u_{\star}^i$ the strong limit in $L^2(\partial\Sigma)$ of $\frac{u_{\eps}^i}{\sqrt{N_i}}$ for $i=1,l$, 
where $M_1 = \left(c_0\right)^2$, $M_l = \left(d_0\right)^2$ and $N_1 + \left(c_0\right)^2 = N_l + \left(d_0\right)^2 = 1$.
Note that this differs from our previous convention, where we did not rescale by the mass.
Recall that $N_1=\frac{1}{2}$.
Similarly, we also have that $N_l$ is bounded away from zero.
This can be seen as follows.
Thanks to \cref{claimassymptotic} and \eqref{eqestthetalinex} we have that 
$$
 \int_{I_\eps^+ \cup I_\eps^-} \frac{\sqrt{t_\eps}y^{-\frac{1}{2}}}{\sqrt{\eps}\sqrt{\ln \frac{1}{r}}} \sin \left(  \pi \frac{\ln y}{\ln r}\right) u_\eps^i dl_\eps
=\sqrt{M_i} +o(1)
$$
as $\eps \to 0$
for $i=1,l$.
From this we conclude that $M_1+M_l \leq 1 + o(1)$, which gives that $N_l \geq \frac{1}{2}-o(1)$.

\bigskip

\begin{claim} \label{claimc_1c_0d_1d_0}
We have the following two alternatives, either 
$$ \sigma_{\star} - \sigma_{\eps}^l = o(\eps) $$
as $\eps\to 0$ or we have the estimates
\begin{equation} \label{eqestonc_1c_0} \eps^{\frac{\alpha-1}{2}}\sqrt{t_\eps} c_1 = \sqrt{1-(c_0)^2 } u_{\star}^1(p_0) +O(\eps^{\frac{1}{2}})  \end{equation}
and
\begin{equation} \label{eqestond_1d_0}  \eps^{\frac{\alpha-1}{2}}\sqrt{t_\eps} d_1 = \sqrt{1-(d_0)^2} u_{\star}^l(p_0) +O(\eps^{\frac{1}{2}}) \end{equation}
as $\eps\to 0$.
\end{claim}

\bigskip

\begin{proof}
We go back to \eqref{eqestneibp0precise} in Claim \ref{claimpointwise}, which holds for $i=1,l$ and gives, after incorporating the estimates from the end of \cref{section5}, that
\begin{equation} \label{eqestneibp0precise2} 
\left\vert \left( u^i_{\eps}\circ \varphi_0^{-1} \right) (x) - \sqrt{N_i}u^i_{\star}(p_0)  \right\vert \leq C \left(\left\| u^i_{\eps} - \sqrt{N_i} u^i_{\star} \right\|_{W^{1,2}(\Sigma)} + \left\| \nabla \phi_i \right\|_{L^{2}(F_{\eps})} + \eps\ln \frac{1}{\eps} \right) 
 \end{equation}
for any $x \in \varphi_0^{-1}\left(\mathbb{D}_{\eps^2}\right)$ where $F_{\eps} = \{(x,y)\in \Omega ; 1-\eps \leq y\leq 1\}$. We also have the estimate \eqref{eqenergydeltaeps2}, which is also true for $i=1,l$,
\begin{equation} \label{eqenergydeltaeps3}
 \frac{\left\| \left(\theta_i\right)_x \right\|_{L^2(\widetilde{\Omega})}^2}{\eps^2}+ \int_0^1 \frac{\overline{\left( \left(\theta_i\right)_v - \overline{\left(\theta_i\right)_v}\right)^2} }{\left(\ln r\right)^2} + \frac{\left(\delta_i\right)_{\eps}}{t_{\eps}} = \frac{\eps}{t_{\eps}}\frac{\overline{\phi_i}(1)^2}{2} + 
 O\left( \eps^{1+ \alpha}\overline{\phi_i}(1)^2 + \eps^{1 + \alpha} \right) 
\end{equation}
as $\eps\to 0$. By the same computations as in \cref{section5}, on $\theta_i$, we know that
\begin{eqnarray*} 
\frac{1}{t_{\eps}} \left\| \nabla \phi_i  \right\|_{L^2\left(F_{\eps} \right)}^2 & = & 
\left( \frac{1}{\eps^2}\int_{0}^{\frac{\ln\left(1-\eps \right)}{\ln r}} \int_{-\frac{r^{2v}}{2}}^{\frac{r^{2v}}{2}} \left(\theta_i\right)_x^2 dx dv + \int_{0}^{\frac{\ln\left(1-\eps \right)}{\ln r}} \overline{\left(\frac{\theta_i}{2} + \frac{\left(\theta_i\right)_v}{\ln\frac{1}{r}} \right)^2} dv\right) 
\\
& \leq & \frac{\left\| \left(\theta_i\right)_x  \right\|_{L^2\left(\widetilde{\Omega}\right)}^2}{\eps^2} + \frac{1}{\left(\ln r\right)^2}\int_{0}^1 \overline{\left(\left(\theta_i\right)_v - \overline{\left(\theta_i \right)_v}\right)^2}dv  \\
& & \quad+ \frac{1}{\left(\ln r\right)^2} \int_{0}^{\frac{\ln\left(1-\eps \right)}{\ln r}} \left(\overline{\left(\theta_i\right)_v}\right)^2 dv + \frac{\left( \overline{\theta^i}^2\left(\frac{\ln\left(1-\eps \right)}{\ln r}\right) - \overline{\theta_i}^2(0)\right)}{2\ln\frac{1}{r}}
\\ 
& & \quad+   \frac{1}{4}\int_{0}^{\frac{\ln\left(1-\eps \right)}{\ln r}} \overline{\theta_i}^2dv + O\left( \left\| \left(\theta_i\right)_x  \right\|_{L^2\left(\widetilde{\Omega}\right)} \right) 
\end{eqnarray*}
as $\eps\to 0$ for $i=1,l$. 
By \eqref{eqenergydeltaeps3} and the asymptotic analysis on $\overline{\theta_i}$ and $\overline{\left(\theta_i\right)_v}$ we have
\begin{equation} \label{eqnablaphiverysmall} \left\| \nabla \phi_i \right\|_{L^2\left(F_{\eps} \right)} = O\left(\eps^{\frac{1}{2}} \right) \end{equation} 
as $\eps\to 0$. Let $v$ be a first eigenfunction associated to $\sigma_{\star}$, bounded in $W^{1,2}(\Sigma)$. 
We have
\begin{equation} \label{eq_R_energy}
\begin{split}
\int_{\Sigma} \left\vert \nabla \left( u^i_{\eps} - v \right) \right\vert^2 
& = \int_{\Sigma} \left\vert \nabla u^i_{\eps} \right\vert^2 + \int_{\Sigma} \left\vert \nabla v\right\vert^2  - 2\int_{\Sigma} \left\langle \nabla  u^i_{\eps} , \nabla v \right\rangle 
\\
& = \left(\delta_i\right)_{\eps} + \sigma_{\eps}\int_{\partial\Sigma} \left(u^i_{\eps}\right)^2 + \sigma_{\star}\int_{\partial\Sigma} v^2 - 2 \sigma_{\star} \int_{\partial\Sigma} v u^i_{\eps} + O\left(\eps^2\ln\frac{1}{\eps}\right) 
\\
& = \left(\delta_i\right)_{\eps} + \left(\sigma^i_{\eps}-\sigma_{\star}\right)\int_{\partial\Sigma} \left(u^i_{\eps}\right)^2 + \sigma_{\star}\int_{\partial\Sigma} \left(v - u^i_{\eps} \right)^2 + O\left(\eps^2\ln\frac{1}{\eps}\right) 
\\
& \leq  C \frac{\eps}{t_{\eps}}\frac{\overline{\phi_i}(1)^2}{2} + \sigma_{\star}\int_{\partial\Sigma} \left(v - u^i_{\eps} \right)^2 +  O\left( \eps^{1+\alpha}\overline{\phi_i}(1)^2 + \eps^{1+\alpha}\right)  \hskip.1cm,
\end{split}
\end{equation}
where the second equality comes from integrating by parts and the pointwise estimate \eqref{eqestneibpi} on $u^i_{\eps}$ and the fourth inequality uses \eqref{eqineqsigmastarordereps} and \eqref{eqenergydeltaeps3}. 
In particular, for $v=\sqrt{N_i} u^i_{\star}$, this means that the right-hand term of \eqref{eqestneibp0precise2} converges to $0$ and that the $W^{1,2}$-norm of $u^i_{\eps} - \sqrt{N_i} u^i_{\star}$ is controled by its $L^2$-norm up to a term of order $\eps^{\frac{1}{2}}$. 

Let us now prove that either $\sigma_{\star} - \sigma^i_{\eps} = o(\eps)$ or $\left\| u^i_{\eps} - \sqrt{N_i} u^i_{\star} \right\|_{L^2(\partial{\Sigma})} = O\left(\eps^{\frac{1}{2}}\right)$ as $\eps\to 0$. 
This will complete the proof of Claim \ref{claimc_1c_0d_1d_0}.
By contradiction, we assume 
\begin{equation}\label{eqabsurdassumption} 
\eps^{\frac{1}{2}} = o\left(\left\| u^i_{\eps} - \sqrt{N_i}u^i_{\star} \right\|_{L^2(\partial{\Sigma})}\right) \text{ and } \eps = O\left( \sigma_{\star}-\sigma_{\eps}^i \right) \text{ as } \eps\to 0 \hskip.1cm. 
\end{equation}
We integrate the equation satisfied by $u_{\eps}^i$ on $\Sigma \setminus \left(\varphi_0^{-1}(\mathbb{D}^+_{\eps^2}) \cup \varphi_1^{-1}\left(\mathbb{D}^+_{r^2 \eps^2}\right)\right)$, 
against a first eigenfunction $v$ associated to $\sigma_{\star}$ and we get thanks to \eqref{eqestneibp0precise2} combined with the estimates above that
$$ 
\left(\sigma_{\star} - \sigma_{\eps}^i \right) \int_{\partial\Sigma} v u_{\eps}^i dl_g + O\left(\eps^2 \sqrt{N_i}\right) = \int_{\varphi_0^{-1}(\mathbb{S}^+_{\eps^2}) \cup \varphi_1^{-1}\left(\mathbb{S}^+_{r^2 \eps^2}\right)} \left( v\partial_{\nu} u_{\eps}^i - u_{\eps}^i\partial_{\nu} v \right) \hskip.1cm. 
$$
By elliptic theory at the scale $\eps^2$ at the neighborhood of $p_0$, we have the following gradient estimate
\begin{equation} \label{eqpointwisegradient} 
\eps^4 \left\vert \nabla  u_{\eps}^i  \right\vert^2(\varphi_0^{-1}(x)) \leq C \left(\eps + \int_{\partial\Sigma} \left(u^i_{\eps} \right)^2 \right) 
\end{equation}
for any $x \in \mathbb{D}^+_{2\eps^2} \setminus \mathbb{D}^+_{\frac{2\eps^2}{3}}$ and at the scale $r^2\eps^2$ at the neighbourhood of $p_1$
\begin{equation} \label{eqpointwisegradient2} 
\eps^4r^4 \left\vert \nabla\left(u_{\eps}^i \right) \right\vert^2(\varphi_1^{-1}(x)) \leq C \left(\eps + \int_{\partial\Sigma} \left( u^i_{\eps} \right)^2 \right) 
\end{equation}
for any $x \in \mathbb{D}^+_{2r^2\eps^2} \setminus \mathbb{D}^+_{\frac{2r^2\eps^2}{3}}$. 
We simply have that
\begin{eqnarray*} 
\left(\sigma_{\star} - \sigma_{\eps}^i\right) \int_{\partial\Sigma} v u_{\eps}^i dl_g  &= &  \int_{\varphi_0^{-1}(\mathbb{S}^+_{\eps^2})} v\partial_{\nu}u_{\eps}^i + \int_{\varphi_1^{-1}(\mathbb{S}^+_{r^2\eps^2})} \left(v-v(p_1)\right)\partial_{\nu}u_{\eps}^i \\
& & \quad+ v(p_1)\int_{\varphi_1^{-1}(\mathbb{S}_{r^2\eps^2})} \partial_{\nu}u_{\eps}^i 
+  O\left(\eps^2 \sqrt{N_i} \right) \hskip.1cm.
\end{eqnarray*}
By integration by parts combined with the pointwise estimates from \cref{claimpointwise} we also have that
$$
 v(p_1)\int_{\varphi_1^{-1}(\mathbb{S}_{r^2\eps^2})} \partial_{\nu}u_{\eps}^i = - v(p_1) \eps r^2 \overline{\phi_y} + O\left( r^2 \eps^2 \log\left( \frac{1}{r\eps} \right) \right)
 = 
 O(r) \text{ as } \eps\to 0\hskip.1cm,
 $$
We also have that $v-v(p_1) = O\left(r^2\eps^2\right)$, uniformly on $\mathbb{S}_{r^2\eps^2}^+$.
If we assume in addition that $v$ satisfies $v(p_0)= 0$, we have that $\vert v \vert = O(\eps^2)$ uniformly on $\mathbb{S}_{\eps^2}^+$.
Therefore, by the uniform estimates \eqref{eqpointwisegradient} and \eqref{eqpointwisegradient2} on the gradient, we have for such $v$ that
\begin{equation} \label{eqveepsi} \left(\sigma_{\star} - \sigma_{\eps}^i\right) \int_{\partial\Sigma} v u_{\eps}^i dl_g = O\left(\eps^2 \sqrt{N_i} \right) \end{equation}
as $\eps\to 0$.

For any closed set $E$ in $L^2\left(\partial\Sigma \right)$, we denote by $P_{E}$ be the orthogonal projection in $L^2$ on $E$. Let $E_{\star}$ be the space generated by the eigenfunctions associated to $\sigma_{\star}$. 
The space $F_{\star} = \{v\in E_{\star} :  v(p_0) = 0\}$, has codimension at most $1$ in $E_{\star}$. 
By  the assumption \eqref{eqabsurdassumption} on the distance of the eigenvalues and \eqref{eqveepsi}, we have that
$$ 
\left\| P_{F_{\star}}(u_{\eps}^i) \right\|_{L^2(\partial\Sigma)} = O\left(\eps \sqrt{N_i}\right)
$$
as $\eps\to 0$. 
Therefore, now by our assumption on the eigenfunction from \eqref{eqabsurdassumption}, we have that $u^i_{\star} \in F_{\star}^{\perp}$.
Moreover, since $\left\|\sqrt{N_i}u^i_{\star}  \right\|_{L^2(\partial{\Sigma})} =  \left\| u^i_{\eps} \right\|_{L^2(\partial{\Sigma})}$ we also have that
$$
 \left\| P_{E_{\star}}\left(u_{\eps}^i\right) - \sqrt{N_i}u^i_{\star} \right\|_{L^2(\partial{\Sigma})} = O\left(\left\| u^i_{\eps} - \sqrt{N_i}u^i_{\star} \right\|_{L^2(\partial{\Sigma})}^2 + \sqrt{N_i} \eps \right) 
$$
as $\eps\to 0$. 
Since by \eqref{eqabsurdassumption}, $\eps = O\left(\left\| u^i_{\eps} - \sqrt{N_i}u^i_{\star} \right\|_{L^2(\partial{\Sigma})}^2\right)$, we have 
$$ 
\left\|  u^i_{\eps} - P_{E_{\star}}\left(u_{\eps}^i\right) \right\|_{L^2(\partial{\Sigma})} \sim \left\| u_{\eps}^i - \sqrt{N_i}u^i_{\star} \right\|_{L^2(\partial{\Sigma})} \text{ as } \eps\to 0\hskip.1cm. 
$$
We have the following equation on $R^i_{\eps} = u_{\eps}^i - P_{E_{\star}}\left(u_{\eps}^i\right)$
\begin{equation} \label{eqontherestonsigma} \begin{cases}
 \Delta_g  R^i_{\eps} = 0 &  \text{in} \ \Sigma
\\
\partial_{\nu} R^i_{\eps} = \sigma_{\eps}R^i_{\eps} - \left(\sigma_{\star}-\sigma^i_{\eps}\right)P_{E_{\star}}\left(u_{\eps}^i\right) & \text{ on } \ \partial\Sigma \setminus \left(A_0\cup A_1\right) \hskip.1cm.  \\
\end{cases} \end{equation} 
We have from \eqref{eq_R_energy} applied to $v =  P_{E_{\star}}\left(u_{\eps}^i\right)$ combined with 
$\|R^i_\eps\|_{L^2(\partial \Sigma)} \gtrsim \eps^{\frac{1}{2}}$ 
that $\frac{R^i_{\eps}}{\left\| R^i_{\eps} \right\|_{L^2(\partial{\Sigma})}}$ is uniformly bounded in $W^{1,2}(\Sigma)$. 
Therefore, we may take
$$
\frac{R^i_{\eps}}{\left\| R^i_{\eps} \right\|_{L^2(\partial{\Sigma})}} 
\to
R^i_{\star}
$$
weakly in $W^{1,2}(\Sigma)$ and strongly in $L^2(\partial \Sigma)$.
By the strong convergence in $L^2(\partial \Sigma)$ we have that $\left\| R^i_{\star} \right\|_{L^2(\partial{\Sigma})} = 1$. 
Since $\sigma_{\star}-\sigma^i_{\eps} = O\left(\eps\right)$ as $\eps\to 0$, by standard elliptic theory on any compact subset of $\Sigma\setminus \{p_0,p_1\}$, we get the following equation at the limit
\begin{equation} \label{eqontherestonsigmastar}
	 \begin{cases}
 		\Delta_g  R^i_{\star} = 0 &  \text{in} \ \Sigma\setminus\{p_0,p_1\}
		\\
		\partial_{\nu} R^i_{\star} = \sigma_{\star}R^i_{\star}  & \text{ on } \ \partial\Sigma \setminus\{p_0,p_1\} \hskip.1cm.  \\
	\end{cases} 
\end{equation} 
Since $R^i_{\star} \in W^{1,2}(\Sigma)$, the equation \eqref{eqontherestonsigmastar} holds on all of $\Sigma$. 
Then, since we have that $R^i_{\star}$ is orthogonal to the eigenspace associated to $\sigma_{\star}$ by construction, 
we must in fact have that $R^{i}_{\star} = 0$. 
This contradicts that $\left\| R^i_{\eps} \right\|_{L^2(\partial{\Sigma})} = 1$.

Therefore, either $\sigma_{\star} - \sigma^i_{\eps} = o(\eps)$ or $\left\| u^i_{\eps} - \sqrt{N_i} u^i_{\star} \right\|_{L^2(\partial{\Sigma})} = O\left(\eps^{\frac{1}{2}}\right)$ as $\eps\to 0$. 
This and \eqref{eqnablaphiverysmall} applied to \eqref{eqestneibp0precise2} completes the proof of \cref{claimc_1c_0d_1d_0}.
\end{proof}

\bigskip

We are now in position to prove the theorem. 
Of course we may assume that
$$
\eps = O(\sigma_\star - \sigma_\eps^l)
$$
since there is nothing to prove otherwise, as we also have that $\sigma_\eps^l - \sigma_\eps^1 = o(\eps)$.
Therefore, we may assume that we have \eqref{eqestonc_1c_0} and \eqref{eqestond_1d_0}.

Recall that the main problem in testing $u_{\eps}^1$ in the variational characterization of the first eigenvalue $\sigma_{\star}$ on $\Sigma_{\star}$ gives an estimate of order $\eps$, as soon as $u_{\star}^1(p_0)\neq 0$ (thanks to the estimate \eqref{eqenergydeltaeps2} and the Poincar\'e inequality). 
We now show that the function $\Psi = u_{\eps}^l + \gamma u_{\eps}^1 $ where $\gamma = - \frac{d_1}{c_1}$ is a better test function, because it is a linear combination of $u_{\eps}^1$ and $u_{\eps}^l$ such that $\overline{\Psi}(1)=0$ for any $\eps$. 
Thanks to \eqref{eqmeanvaluephi}, we can compute the asymptotic expansion of the mean value of $\Psi$ on $\partial \Sigma$
\begin{equation} \label{eqassympexpangamma} 
\int_{\partial \Sigma} \Psi dl_g = - \left(\gamma \int_{I_\eps^+\cup I_\eps^-} \phi_1 dl_{\eps} + \int_{I_\eps^+\cup I_\eps^-}\phi_l dl_{\eps} \right) + O(\eps^2) =  O\left(\eps^{\frac{3\alpha}{2}+\frac{1}{2}}\right)
\end{equation}
as $\eps\to 0$. 
Testing $\Psi$ in the variational characterization of the first non-zero eigenvalue $\sigma_{\star}$ on $\Sigma$, 
we then have that
$$
\sigma_{\star} \leq \frac{\int_{\Sigma} \left\vert \nabla \Psi \right\vert_g^2 dA_g}{\int_{\partial\Sigma}\Psi^2 dl_g - \left(\int_{\partial \Sigma} \Psi dl_g\right)^2} =  \frac{\gamma^2 \sigma_{\eps}^1 +  \sigma_{\eps}^l - \int_{\Omega} \left\vert \nabla \Psi \right\vert_{g_{\eps}}^2 dA_{g_{\eps}}}{\gamma^2 + 1 - \int_{I^+ \cup I^-}\Psi^2 dl_{g_{\eps}} + O\left( \eps^{2}\right)}  \hskip.1cm, 
$$
where
$$
\int_{I^+\cup I^-}\Psi^2 dl_{g_{\eps}} + \int_{\partial\Sigma}\Psi^2 dl_g = \gamma^2+1 + O(\eps^2)
$$
and
$$
 \int_{\Sigma} \left\vert \nabla \Psi \right\vert_g^2 dA_g + \int_{\Omega} \left\vert \nabla \Psi \right\vert_g^2 dA_g = \gamma^2  \sigma_{\eps}^1 + \sigma_{\eps}^l
 $$
because $u_{\eps}^1$ and $u_{\eps}^l$ have unit $L^2$-norm and are orthogonal in $L^2(\partial \Sigma_{\eps})$, 
we then get
\begin{equation} \label{eqestsigmastar}
\sigma_{\star} \leq \frac{\gamma^2 \sigma_{\eps}^1 +  \sigma_{\eps}^l}{1+\gamma^2} + \frac{\left(\gamma^2\sigma_{\eps}^1 +  \sigma_{\eps}^l\right)\int_{I^+ \cup I^-}\Psi^2 dl_{g_{\eps}} - \left(\gamma^2+1\right) \int_{\Omega} \left\vert \nabla \Psi \right\vert_{g_{\eps}}^2 dA_{g_{\eps}}}{\left(\gamma^2+1\right)\left(\gamma^2+1 - \int_{I^+ \cup I^-}\Psi^2 dl_{g_{\eps}}\right) + O\left( \eps^{2}\right)} \hskip.1cm.
\end{equation}
We set
$$ 
A = \left(\gamma^2 \sigma_{\eps}^1 + \sigma_{\eps}^2\right)\int_{I^+ \cup I^-}\Psi^2 dl_{g_{\eps}} - \left(\gamma^2 +1\right) \int_{\Omega} \left\vert \nabla \Psi \right\vert_{g_{\eps}}^2 dA_{g_{\eps}} 
$$
and
$$ 
B = \gamma^2 +1 - \int_{I^+ \cup I^-}\Psi^2 dl_{g_{\eps}}  \hskip.1cm.
$$
We aim at getting an upper bound on $A$ and a lower bound on $B$. 
Let $\Theta$ be defined by
$$ 
\Psi(x,y) = \frac{\sqrt{t_{\eps}}y^{-\frac{1}{2}}}{\sqrt{\eps}\sqrt{\ln\frac{1}{r}}} \Theta \left(x, \frac{\ln(y)}{\ln(r)} \right) \hskip.1cm.
$$
By our computations in the preceeding section (see \eqref{eqenergyeigenfunc2}, \eqref{eqenergyeigenfunc3}, \eqref{eqenergyeigenfunc4} and \eqref{eqenergyeigenfunc5}), we have that 
\begin{equation} \label{eqL2normofPsi} 
\int_{I^+ \cup I^-}\Psi^2 dl_{g_{\eps}} = \int_{0}^1 \left(\Theta^2\left(\frac{r^{2v}}{2},v\right)+\Theta^2\left(-\frac{r^{2v}}{2},v\right)\right) = 2 \int_{0}^1\overline{\Theta}^2 dv + O\left(\left\| \Theta_x \right\|_{L^2(\widetilde{\Omega})} \right) 
\end{equation}
and that
\begin{equation}\label{eqenergyofPsi}
 \begin{split} 
 \int_{\Omega} \left\vert \nabla \Psi \right\vert_{g_{\eps}}^2 dA_{g_{\eps}} = t_{\eps}\left(\frac{1}{4} \int_{0}^1 \overline{\Theta}^2 dv + \frac{1}{\left(\ln r\right)^2}\int_0^1 \overline{\Theta_v}^2dv - \frac{\overline{\Theta}(0)^2}{2\ln\frac{1}{r}}\right) \\ + t_{\eps} \left( \frac{\int_{\widetilde{\Omega}} \Theta_x^2}{\eps^2}  + \frac{1}{\left(\ln r\right)^2}\int_{0}^{1} \overline{\Theta_v^2 - \overline{\Theta_v}^2}dv + e_{\eps} \right) \hskip.1cm,
\end{split}
\end{equation}
where $e_{\eps}$ is a very small error term compared to $\frac{\int_{\widetilde{\Omega}} \Theta_x^2}{\eps^2}$. 
Therefore, we have that
\begin{equation}\label{eqestnumerator} 
A \leq \left(\gamma^2\left(2\sigma_{\eps}^1 - \frac{t_{\eps}}{4}\right) + \left(2\sigma_{\eps}^2 - \frac{t_{\eps}}{4}\right)\right) \int_{0}^1 \overline{\Theta}^2 dv - \frac{(\gamma^2+1)t_{\eps}}{\left(\ln r\right)^2} \int_0^1 \overline{\Theta_v}^2dv \hskip.1cm,
 \end{equation}
where we recall that $\gamma$ was defined so that $\overline{\Theta}(0) = 0$
Since $\left\| \Theta_x \right\|_{L^2(0,1)}=O\left( \eps \right)$
we also have that
\begin{equation} \label{eqestdenominator} 
B = \gamma^2+1 - 2 \int_{0}^1\overline{\Theta}^2 dv + O\left(\eps \right)
\end{equation}
as $\eps\to 0$. 
Because of the asymptotic expansions of $\sigma_{i}^{\eps}$, $\bar{\theta^i}$, and $\overline{\theta_{v}^i}$ for $i=1,l$, we get 
$$ 
\int_{0}^1 \overline{\Theta}^2 dv = \left(e_0\right)^2 \int_{0}^1 f^2 + 2 e_0 e_2 \int_{0}^1 f f_2 + O\left( \left\vert e_3e_0 \right\vert \right)\hskip.1cm,
$$
$$ 
\int_{0}^1 \overline{\Theta_v}^2 dv = \left(e_0\right)^2 \int_{0}^1 \left(f'\right)^2  + 2 e_0 e_2 \int_{0}^1 f' f_2' + O\left( \left\vert e_3e_0 \right\vert  \right) \hskip.1cm,
$$
$$ 
\left(\gamma^2\left(2\sigma_{\eps}^1 - \frac{t_{\eps}}{4}\right) + \left(2\sigma_{\eps}^l - \frac{t_{\eps}}{4}\right)\right) = \frac{t_{\eps}}{\left(\ln r\right)^2}\left( \pi^2 (1+\gamma^2) - 2\pi \tilde{e}_1 \right) + O\left( \tilde{e}_3 \right) \hskip.1cm,
$$
as $\eps\to 0$, where 
$$ 
e_0 = \gamma c_0 +  d_0 \hskip.1cm,\hskip.1cm e_1 = \gamma c_1 +  d_1 = 0  \hskip.1cm,\hskip.1cm e_2 = \gamma\frac{ \left(c_1\right)^2}{c_0}+ \frac{\left(d_1\right)^2}{d_0} \hskip.1cm,\hskip.1cm e_3 = \gamma\frac{ \left(c_1\right)^3}{\left(c_0\right)^2}+ \frac{\left(d_1\right)^3}{\left(d_0\right)^2} \hskip.1cm,
$$
$$ 
\tilde{e}_1 = \gamma^2\frac{c_1}{c_0} + \frac{d_1}{d_0} \hskip.1cm,\hskip.1cm \tilde{e}_3 = \gamma^2\left(\frac{c_1}{c_0}\right)^3 +  \left(\frac{d_1}{d_0}\right)^3 \hskip.1cm.
$$ 
Using $\int_0^1 f f_1 dv = 0$ when we integrate $f_2''+\pi^2 f_2 = 2\pi f_1$, we have that
\begin{equation} \label{eqestl2barTheta} 
\int_{I^+ \cup I^-}\overline{\Theta}^2 dl_{g_{\eps}} = \frac{\left(e_0\right)^2}{2} + 2 e_0 e_2 \int_0^1 f_2 f  + O\left(\eps^{\frac{3}{2}-\frac{3\alpha}{2}} \right) 
\end{equation}
\begin{equation} \label{eqestl2barThetav} 
\int_{I^+ \cup I^-}\overline{\Theta_v}^2 dl_{g_{\eps}} = \pi^2 \left( \frac{ \left(e_0\right)^2}{2} + 2 e_0 e_2 \int_0^1 f_2 f \right) + O\left(\eps^{\frac{3}{2}-\frac{3\alpha}{2}} \right) \hskip.1cm,
\end{equation}
so that, \eqref{eqestnumerator} becomes
\begin{equation}\label{eqestnumerator2} 
A \leq \frac{t_{\eps}}{\left(\ln r\right)^2} \left( - 2\pi \tilde{e}_1\frac{ \left(e_0\right)^2}{2} + O\left( \tilde{e}_1 e_0 e_2 \right) + O\left( \eps^{\frac{3}{2}-\frac{3\alpha}{2}} \right) \right) = O\left(\eps^{\frac{3\alpha}{2}+\frac{1}{2}}\right) \hskip.1cm.
 \end{equation}
By the integral formula $-2\int_{0}^1 f f_2 = \int_0^1 \left(f_1\right)^2 \neq 0 $ and using \eqref{eqestl2barTheta}, the estimate \eqref{eqestdenominator} on $B$ becomes
\begin{equation} \label{eqestdenominator2} 
B = \gamma^2 +1 - \left(e_0\right)^2 + 2 e_0 e_2 \int_0^1 \left(f_1\right)^2 + o\left(\eps^{1-\alpha} \right) \hskip.1cm.
\end{equation}
We recall that 
$$
e_0 = d_0\left(1 - \frac{d_1 c_0}{c_1 d_0} \right) \hskip.1cm,\hskip.1cm \tilde{e}_1 = \frac{d_1}{d_0}\left( 1+ \frac{d_1 d_0 }{c_1 c_0} \right) \hskip.1cm,\hskip.1cm e_2 = \frac{\left(d_1\right)^2}{d_0}\left( 1- \frac{c_1 d_0 }{d_1 c_0} \right) \hskip.1cm.
$$
From \eqref{eqestdenominator2}, we have the following lower bound on $B$ 
\begin{eqnarray*} 
B & = & 1 - 2\left(d_0\right)^2 + \left(\frac{\gamma}{\sqrt{2}}-d_0\right)^2 + 2 e_0 e_2 \int_0^1 \left(f_1\right)^2 +  o\left(\eps^{1-\alpha}\right) \\
 &\geq & 1 - 2\left(d_0\right)^2 + 2\int_0^1 \left(f_1\right)^2\left(d_1\right)^2\left(1-\frac{d_1 c_0}{c_1 d_0}\right) \left( 1- \frac{c_1 d_0 }{d_1 c_0} \right) + o\left(\eps^{1-\alpha}\right)  
\end{eqnarray*}
as $\eps\to 0$. 
Now, using that $u_{\eps}^1$ and $u_{\eps}^l$ are orthogonal in $L^2(\partial\Sigma_{\eps})$, we get
\begin{equation} \label{eqorthogonalustar}
\sqrt{1 - \left(c_0\right)^2}\sqrt{1 - \left(d_0\right)^2} \int_{\partial \Sigma} \frac{u_{\eps}^1}{\sqrt{N_l}} \frac{u_{\eps}^l}{\sqrt{N_2}} dl_g + O(\eps^2) = - 2 \int_{0}^1 \overline{\theta_1} \overline{\theta_l} dv   
\end{equation}
so that using again the asymptotic expansion of $\overline{\theta_1}$ and $\overline{\theta_l}$ in $L^2$, we have that
\begin{equation} \label{eqorthogonalustar2}
\sqrt{1 - \left(c_0\right)^2}\sqrt{1 - \left(d_0\right)^2} \int_{\partial \Sigma} \frac{u_{\eps}^1}{\sqrt{N_1}} \frac{u_{\eps}^l}{\sqrt{N_2}} dl_g  = - c_0 d_0 + o(1)
\end{equation}
as $\eps \to 0$. 
We recall that we chose $t_{\eps}$ such that $\left(c_0\right)^2 = M_1 = \frac{1}{2}$ and since $u_{\star}^1$ and $u_{\star}^2$ have a unit $L^2$-norm, we have that
$$ 2 d_0^2 \leq  1 + o(1) \text{ as } \eps \to 0 \hskip.1cm,$$
so that $B$ converges to $0$ if and only if $\left(d_0\right)^2 \to \frac{1}{2}$. 
If $\left(d_0\right)^2$ does not converge to $\frac{1}{2}$, then, 
since by \eqref{eqestnumerator2}, $A=O\left(\eps^{\frac{3\alpha}{2}+\frac{1}{2}}\right)$ and $B$ is bounded from below by a positive constant, we get the theorem.\footnote{In fact, it is not hard to show this case does not occur.}

We assume now that $\left(d_0\right)^2 \to \frac{1}{2}$ as $\eps\to 0$. Thanks to \eqref{eqestonc_1c_0} and \eqref{eqestond_1d_0}, and since $\left(c_0\right)^2 = \frac{1}{2}$ and $\left(d_0\right)^2 = \frac{1}{2} +o(1)$, passing to the limit on \eqref{eqorthogonalustar2} gives, 
\begin{equation} \label{eqintu1staru2star} 
\int_{\partial \Sigma} u_{\star}^1u_{\star}^l dl_g = - 1  \hskip.1cm,
\end{equation} 
where $u_{\star}^i$ is the strong limit in $L^2(\partial\Sigma)$ of $\frac{u_{\eps}^i}{\sqrt{N_i}}$ for $i=1,l$. 
By \eqref{eqintu1staru2star}, since $u_{\star}^1$ and $u_{\star}^l$ have unit norm, then $u^{1}_{\star} = - u_{\star}^l$, so that $\gamma = -\frac{d_1}{c_1} = 1 + o\left( 1 \right)$ as $\eps\to 0$. 
We finally get the following lower bound on $B$ given by
\begin{equation} \label{eqlastonB} 
B  \geq  1 - 2\left(d_0\right)^2 + 8\int_0^1 \left(f_1\right)^2 \left(d_1\right)^2 +  o\left(\eps^{1-\alpha}\right)  
\end{equation}
as $\eps\to 0$. 
At the same time, we have the following estimate on $A$ given by
\begin{equation} \label{eqlastonA} 
A \leq \frac{t_{\eps}}{\left(\ln r\right)^2} \left( 4\pi  \frac{d_1}{d_0}\left( 1+ \frac{d_1 d_0 }{c_1 c_0}  \right) +O\left(\eps^{\frac{3}{2} -\frac{3\alpha}{2} } \right) \right) \leq C \eps^{2\alpha} \frac{d_1}{d_0} \left( 1+ \frac{d_1 d_0 }{c_1 c_0}  \right) 
 \end{equation}
 as $\eps \to 0$, for some constant $C$ independent from $\eps$. 
 Thanks to \eqref{eqestonc_1c_0} and \eqref{eqestond_1d_0}, since we know that $u_{\star}(p_0):= u^{1}_{\star}(p_0) = - u_{\star}^2(p_0)$, we have that
$$ 
c_1 + d_1 = \eps^{\frac{1-\alpha}{2}}\frac{u_{\star}(p_0)}{\sqrt{t_{\eps}}} \left( \sqrt{1-(c_0)^2 } -  \sqrt{1-(d_0)^2 } \right) +O\left(\eps^{1 - \frac{\alpha}{2}} \right)  
$$
as $\eps\to 0$, so that since $\left(c_0\right)^2 = \frac{1}{2}$ and $d_0 - c_0 = o(1)$, we also have that
$$ 
c_1 + d_1 = O\left( \eps^{\frac{1-\alpha}{2}} \left\vert c_0 - d_0\right\vert \right) + O\left(\eps^{1-\frac{\alpha}{2}}\right) 
$$
as $\eps\to 0$. 
By the previous formula and since $d_1 = - c_1 + o(c_1)$ is of order $\eps^{\frac{1-\alpha}{2}}$, this impliest
\begin{equation} \label{eqc_1d_1c_0d_0} 
\frac{d_1}{d_0}\left( 1+ \frac{d_1 d_0 }{c_1 c_0}  \right) = O\left( \eps^{\frac{1-\alpha}{2}} \left\vert c_0 - d_0\right\vert \right) + O\left(\eps^{1-\frac{\alpha}{2}}\right) \hskip.1cm. 
\end{equation}
Gathering \eqref{eqlastonB}, \eqref{eqlastonA}, \eqref{eqc_1d_1c_0d_0}, and knowing that $\sigma_{\eps}^2 - \sigma_{\eps}^1 = O\left(\eps^{\frac{3\alpha}{2}+\frac{1}{2}}\right)$ as $\eps\to 0$, the inequality \eqref{eqestsigmastar} simply becomes
$$ \sigma_{\star} \leq \sigma_{\eps}^1 + O\left(\eps^{\frac{3\alpha}{2}+\frac{1}{2}}\right) + O\left( \eps^{\frac{5\alpha}{2}} \right)$$
as $\eps\to 0$. 
Choosing $\alpha$ such that $\frac{2}{5}<\alpha<\frac{1}{2}$, we get that $\sigma_{\star} - \sigma_{\eps}^1 = o(\eps)$ as $\eps\to 0$, and we can conclude \cref{thm_glue}.


\bibliographystyle{alpha}
\bibliography{mybibfile}

\nocite{*}
\begin{thebibliography}{GWZ}
%
\bibitem[Ann87]{anne}
C.~Ann{\'e}, {\em Spectre du laplacien et {\'e}crasement d'anses}, Ann.\ Sci.\ {\'E}cole Norm.\ Sup. (4) {\bf 20}, 1987, 271--280.
%
%
\bibitem[CFS20]{CFS}
A.~Carlotto, G.~Franz, M.~B. Schulz {\em Free boundary minimal surfaces with connected boundary and arbitrary genus}, arXiv preprint 2020, arXiv:2001.04920, 16pp.
%
\bibitem[CES03]{ces}
B.~Colbois, A.~El Soufi,
{\em Extremal eigenvalues of the {L}aplacian in a conformal class of metrics: the `conformal spectrum'},
Ann.\ Global Anal.\ Geom.\, {\bf 24}, 2003, no.4, 337--349.
%
\bibitem[Cou40]{courant}
R.~Courant, 
{\em The existence of minimal surfaces of given topological structure under prescribed boundary conditions}, Acta Math. {\bf 72}, 1940, 51–-98.
%
\bibitem[ESI00]{Ilias_ElSoufi}
A.~El Soufi, S.~Ilias, {\em Riemannian manifolds admitting isometric immersions by their first eigenfunctions}, Pacific J.\ Math. {\bf 195}, 2000, 91--99.
\bibitem[FS16]{fs}
A.~Fraser, R.~Schoen, {\em Sharp eigenvalue bounds and minimal surfaces in the ball}, Invent.\ Math. {\bf 203}, 2016, 823--890.
\bibitem[FS19]{fs2}
A.~Fraser, R.~Schoen, {\em Some results on higher eigenvalue optimization}, arXiv preprint 2019, arXiv:1910.03547, 25pp.
%
\bibitem[FPZ17]{FPZ}
A.~Folha, F.~Pacard, T.~Zolotareva,
{\em Free boundary minimal surfaces in the unit 3-ball},
Manuscripta Math. {\bf 154}, 2017, 359--409.
%
\bibitem[GL20]{GL}
A.~Girouard, J.~Lagac\'e, {\em Large Steklov eigenvalues via homogenisation on manifolds},
arXiv preprint 2020, arXiv:2004.04044, 30pp.
\bibitem[Hil85]{Hil}
S.~Hildebrandt, {\em Free boundary problems for minimal surfaces and related questions}, Comm. Pure Appl. Math. {\bf 39} 1986, no. S, suppl., S111–-S138, Frontiers of the mathematical sciences: 1985 (New York, 1985). 
\bibitem[KL17]{kl}
N.~Kapouleas, M.M-C.~ Li,
{\em Free Boundary Minimal Surfaces in the Unit Three-Ball via Desingularization of the Critical Catenoid and the Equatorial Disk},
arXiv preprint 2017, arXiv:1709.08556, 45 pp.
%
\bibitem[KW17]{kw}
N.~Kapouleas, D.~Wiygul,
{\em Free-boundary minimal surfaces with connected boundary in the $3$-ball by tripling the equatorial disc},
arXiv preprint 2017, arXiv:1711.00818, 33 pp.
%
\bibitem[KKP14]{KKP}
M.A. Karpukhin, G. Kokarev, I. Polterovich, {\em Multiplicity bounds for {S}teklov eigenvalues on {R}iemannian surfaces}, Ann. Inst. Fourier (Grenoble), Universit\'e de Grenoble. Annales de l'Institut Fourier, {\bf 64}, 2014, 6, 2481--2502.
\bibitem[KS20]{KS}
M.~Karpukhin, D.~L. Stern {\em Min-max harmonic maps and a new characterization of conformal eigenvalues},
arXiv preprint 2020, arXiv:2004.04086, 59 pp.
\bibitem[Ke17]{ketover_1}
D.~Ketover,
\emph{Free boundary minimal surfaces of unbounded genus},
arXiv preprint 2016, arXiv:1612.08691, 32 pp.
%
\bibitem[Ke17a]{ketover_2}
D.~Ketover,
\emph{Equivariant min-max theory.}
arXiv preprint 2016, arXiv:1612.08692, 42 pp.
%
\bibitem[Kok14]{kokarev}
G.~Kokarev, {\em Variational aspects of {L}aplace eigenvalues on {R}iemannian surfaces}, Adv.\ Math.\ {\bf 258}, 2014, 191--239.
%
\bibitem[Li20]{li}
M.~M-C. Li, {\em Free boundary minimal surfaces in the unit ball: recent advances and open questions}, to appear in Proceedings of the first annual meeting of the ICCM, 2020.
%
\bibitem[MS17]{MS1}
H.~Matthiesen, A.~Siffert,
{\em Existence of metrics maximizing the first eigenvalue on non-orientable surfaces}, to appear in Journal of Spectral Theory, 2017, 14pp.
%
\bibitem[MS19]{MS2}
H.~Matthiesen, A.~Siffert,
{\em Sharp asymptotics for the first eigenvalue on some degenerating surfaces}, to appear in Trans. Amer. Math. Soc, 2019, 35pp.
%
\bibitem[MS19a]{MS3}
H.~Matthiesen, A.~Siffert,
{\em Handle attachement and the normalized first eigenvalue}, arXiv preprints 2019, arXiv:1909.03105v2, 65pp.
%
\bibitem[Nad96]{nadirashvili}
N.~Nadirashvili, {\em Berger's isoperimetric problem and minimal immersions of surfaces}, Geom.\ Func.\ Anal.\ {\bf 6}, 1996, 877--897.
%
\bibitem[NT08]{NT}
S. A.~Nazarov, J.~Taskinen {\em On the spectrum of the Steklov problem in a domain with a peak}, Vestnik St. Petersburg University: Mathematics volume {\bf 41}, 2008, 45–-52.
%
\bibitem[Pet14]{petrides}
R.~Petrides, {\em Existence and regularity of maximal metrics for the first Laplace eigenvalue on surfaces}, Geom.\ Funct.\ Anal.\ {\bf24}, 2014, 1336--1376.
%
\bibitem[Pet18]{petrides-2}
R.~Petrides, {\em  On the existence of metrics which maximize Laplace eigenvalues on surfaces}, Int.\ Math.\ Res.\ Not. , {\bf14}, 2018, 4261--4355.
%
\bibitem[Pet19]{petrides-3}
R.~Petrides, {\em  Maximizing {S}teklov eigenvalues on surfaces}, J. Differential Geom.
Volume {\bf 113},  2019, no.1, 95--188.
%
\bibitem[Rob11]{robin}
N.~Robin,
\emph{Regularity of solutions of linear second order elliptic and parabolic boundary value problems on {L}ipschitz domains},
J.\ Differential Equations {\bf 251}, 2011, 860--880.
%
\end{thebibliography}

\end{document}